\documentclass[11pt,reqno]{amsart}
\usepackage{mathrsfs}
\usepackage{hyperref}
\usepackage{amssymb,amsmath,amsfonts,latexsym}
\usepackage{bm}
\usepackage{mathtools}
\usepackage{enumerate}
\usepackage{array,graphics,color}
\allowdisplaybreaks

\numberwithin{equation}{section}
\newtheorem{dfn}{Definition}[section]
\newtheorem{thm}[dfn]{Theorem}
\newtheorem{lma}[dfn]{Lemma}

\newtheorem{ppsn}[dfn]{Proposition}
\newtheorem{crlre}[dfn]{Corollary}

\newtheorem{prob}{Question}
\newtheorem{rmrk}[dfn]{Remark}

\DeclarePairedDelimiterX{\norm}[1]{\lVert}{\rVert}{#1}
\DeclarePairedDelimiterX{\bnorm}[1]{\big\lVert}{\big\rVert}{#1}
\DeclarePairedDelimiterX{\Bnorm}[1]{\Big\lVert}{\Big\rVert}{#1}

\newcommand{\Z}{\mathbb{Z}}
\newcommand{\T}{\mathbb{T}}

\newcommand{\D}{\mathbb{D}}
\newcommand{\C}{\mathbb{C}}

\newcommand{\hcl}{\mathcal{H}}

\begin{document}
	
	
	\title[Commutants and Complex Symmetry of $M_{B}$]{Commutants and Complex Symmetry of Finite Blaschke Product Multiplication Operator in $\bm{L^2(\T)}$}
	
	\author[Chattopadhyay] {Arup Chattopadhyay}
	\address{Department of Mathematics, Indian Institute of Technology Guwahati, Guwahati, 781039, India}
	\email{arupchatt@iitg.ac.in, 2003arupchattopadhyay@gmail.com}

	\author[Das]{Soma Das}
	\address{Department of Mathematics, Indian Institute of Technology Guwahati, Guwahati, 781039, India}
	\email{soma18@iitg.ac.in, dsoma994@gmail.com}

	
	
	\begin{abstract}
	Consider the multiplication operator $M_{B}$ in $L^2(\T)$, where the symbol $B$ is a finite Blaschke product. In this article, we characterize the commutant of $M_{B}$ in $L^2(\T)$, noting the fact that $L^2(\T)$ is not an RKHS. As an application of this characterization result, we explicitly determine the class of conjugations commuting with $M_{z^2}$ or making $M_{z^2}$ complex symmetric by introducing a new class of conjugations in $L^2(\T)$. Moreover, we analyze their properties while keeping the whole Hardy space, model space, and Beurling-type subspaces invariant. Furthermore, we extended our study concerning conjugations in the case of finite Blaschke.

 \end{abstract}
	\maketitle
	
	\section{Introduction}
	 Suppose $T$ is a bounded linear operator on a separable Hilbert space $\hcl$. Then the commutant of $T$ is usually denoted by $\{T\}^\prime$ and defined by $\{T\}^\prime :=\{S\in \mathcal{B}(\hcl): ST=TS\}$, where $\mathcal{B}(\hcl)$ is the space of all bounded linear operators on $\hcl$. Given a bounded linear operator $T$, it is always interesting to find the commutant  $\{T\}^\prime$ since it plays an important role in determining the structure of the operator. Another strong reason for identifying the commutant of an operator $T$ is to recognize all reducing subspaces of $T$. In this direction, there are various nice works available in the literature, specially for analytic multiplication operators on some well-known function spaces like the Hardy space, the Bergman space, and also for reproducing kernel Hilbert spaces. The detailed can be found in \cite{DoPuWa12, GuHu15, GuSuZhZh09, StZh02, Zh00}.

Let $\mathbb{D}=\{z\in \mathbb{C}:~|z|<1\}$ denote the unit disc in the complex plane and $\mathbb{T}=\{z\in \mathbb{C}:~|z|=1\}$ be the unit circle. Let us denote by  $L^2 (\mathbb{T})$ (or simply $L^2$) the space of all square-integrable functions over the unit circle $\T$. The Banach space of all essentially bounded functions on the unit circle $\T$ is denoted by $L^\infty(\T)$ (or $L^\infty$ in short). Therefore, if $f\in L^2(\mathbb{T})$, then the function $f$ can be expressed as $$f(z)=\sum_{n=-\infty}^{\infty}\hat{f}(n)z^n,$$ where $\hat{f}(n)$ denotes the $n$-th Fourier coefficient of $f$ and $\sum\limits_{n=-\infty}^{\infty}|\hat{f}(n)|^2 <\infty$. For $\phi\in L^{\infty}$, let $M_{\phi}$ denote the multiplication operator on $L^2$ defined as
$M_{\phi}(f)=\phi f,\quad f\in L^2.$
Recall that the classical Hardy space over the unit disc $\mathbb{D}$ is denoted by $H^2(\mathbb{D})$ and defined by
\begin{equation*}
	H^2(\mathbb{D}):=\Big\{f(z)=\sum_{n=0}^{\infty}a_nz^n:~\|f\|_{H^2(\mathbb{D})}^2:=\sum_{n=0}^{\infty}|a_n|^2,~z\in \mathbb{D},~a_n\in \mathbb{C}\Big\}.
\end{equation*}
For $f\in H^2(\D)$, the non-tangential boundary limit (radial limit) $\tilde{f}(e^{i\theta})=\lim_{r\to 1-}f(re^{i\theta})$ exists almost everywhere on $\T$ \cite{EMV1,EMV2,RSN}.
Therefore $H^2(\D)$ can be embedded isometrically as a closed subspace of $L^2(\mathbb{T})$ by identifying $H^2(\D)$ through the nontangential boundary limits of the $H^2(\D)$ functions. In other words, we can recognize the Hardy space $H^2(\D)$ as a closed subspace of $L^2(\mathbb{T})$, say $H^2$, consisting of $L^2$ functions having negative Fourier coefficients zero. We denote by $H^{\infty}$ the space of all bounded analytic functions on $\D$ that can be identified with a closed subspace of $L^{\infty}$, that is $H^\infty:= L^\infty \cap H^2$. An inner function $\theta$ is an element in $H^\infty$ such that $|\theta | =1$ almost everywhere on $\T$. For $\phi\in L^{\infty}$, define the Toeplitz operator $T_{\phi}:H^2\to H^2$ by $T_{\phi}f=P(\phi f),$ where $P$ is an orthogonal projection from $L^2$ onto $H^2$. Recall that the unilateral shift $S$ on $H^2$ is defined by $Sf(z)=zf(z), f\in H^2$, a restriction of the bilateral shift $M_z$ on $L^2$. A well-known theorem due to $Beurling$ characterizes all the $S$- invariant subspaces of $H^2$. In other words, any $S$-invariant subspace of $H^2$ is of the form $\theta H^2$, where $\theta$ is an inner function. 
For any inner function $\theta$, the model space $K_\theta$ is defined by $K_\theta := H^2\ominus \theta H^2$, which is $S^*$ invariant. For more on model spaces and related stuff, we refer the reader to \cite{EMV1,EMV2,GMR2016,RSN}. 

The commutant of $S$ in $H^2$ is $\{T_{\phi}:~\phi\in H^{\infty}\}$ and the commutant of $M_z$ in $L^2$ is $\{M_{\phi}:~\phi\in L^{\infty}\}$ \cite{RSN},\cite{RH1973}. In general, it is highly-non trivial to describe the commutant of a given bounded linear operator. In this direction, the commutant of more general shift operators and  Toeplitz operators in various function spaces has been studied by many authors (see \cite{ACZ,CC78,CC80,CD78,DW,KR07,KR01,N67,SV}).

For a finite Blaschke product $B$, let $M_B$ be the operator of multiplication by $B$ in $L^2(\T)$. The commutants of finite
Blaschke product Toeplitz operators $T_B$ are ‘well known’ for the Hardy space $H^2$ \cite{CC78}. So in some sense, we can say that these are the ‘same’ operators that commute with $M_B$ on other Hilbert spaces. In 2015 \cite{CoWa15}, Cowen and Wahl gave a detailed independent description of the operators on the Hilbert space $\hcl$ of analytic functions defined on $\D$ that commutes with $M_B$. It is worth mentioning that the Hilbert space $\hcl$ in their work is a reproducing kernel Hilbert space (RKHS) (see \cite{PaRa16}). Abkar, Cao, and Zhu \cite{ACZ} gives a complete description of the commutant of $M_{z^n}$ on a family of analytic function spaces on $\D$, which include the Bergman space, the Hardy space, and the Dirichlet space. Note that these are shift operators of higher multiplicity, in other words, special multiplication operator by particular finite Blaschke. More explicitly, the authors in \cite{ACZ} provide the description regarding the commutant of $M_{z^n}$ by identifying with $n$ copies of the
multiplier algebra of the underlying function spaces. Recently, Gallardo and Partington \cite{GaPa22} provide a characterization of the commutant of $T_B$  extending the work on the weighted Bergman spaces. 

In this direction, it is worth noting that $L^2(\T)$ is an important function space having various nice properties, but it is not an RKHS \cite{PaRa16}. Also, observe that $L^2(\T)$ does not come as a particular example of the above-mentioned spaces where the characterization of the commutant of $M_B$ has been studied. Our primary goal in this paper is to characterize the commutant of $M_B$ independently in the $L^2(\T)$ -space and obtain a detailed description of the commutant in terms of $L^\infty$- functions. As a consequence, by using the precise structure of the commutator, we will obtain some applications toward characterizing conjugations in $L^2(\T)$.

Now the study of conjugations and complex symmetric operators has become a widely-interested research area because of its significance in the fields of operator theory, complex analysis, and physics. Recently, several authors have been interested in non-Hermitian quantum mechanics and the spectral analysis of certain complex symmetric operators (for more details, see \cite{BFG, GPP2014}). One of the motivations for studying conjugations comes from the study of complex symmetric operators. The study of complex symmetric (in short, C-symmetric) operators was initiated by Garcia, Putinar, and Wogen in \cite{GP2005, GP2007, GW2009, GW2010}. Technically speaking, a conjugation $C$ is an antilinear and isometric involution on $\hcl$. In other words, a conjugate-linear operator $C:\hcl \to \hcl$ is said to be a conjugation if it satisfies the following two conditions:
\begin{enumerate}[(a)]
	\item $C^2 =I_{\hcl}$, (the identiy operator on $\hcl$)
	\item $\langle Cf,Cg \rangle = \langle  g,f \rangle , \forall f,g \in \hcl$.
\end{enumerate}
It is interesting to observe that the conjugation is a straightforward generalization of the conjugate-linear map $z\longrightarrow \bar{z}$ on the one-dimensional Hilbert space $\mathbb{C}$. Moreover, we can also interpret conjugation as a special kind of antilinear operator and note that the antilinear operator is the only type of non-linear operator that is important in the field of quantum mechanics (for more details, see \cite{Mess, SL1996, Sw1995}). In this connection, it is remarkable to mention that Garcia and Putinar have shown in \cite{GP2005} that for any given conjugation $C$ on $\hcl$, there exists an orthonormal basis $\{e_n: n\in \mathbb{N}_0\}$ such that $Ce_n=e_n$ for all $n\in \mathbb{N}_0$, where $\mathbb{N}_0$ denotes the set of all non-negative integers.
\begin{dfn}
	An operator $T\in \mathcal{B}(\mathcal{H})$ is called  $C$-symmetric if there exists a conjugation $C$ on $\mathcal{H}$  such that  $CTC=T^*$, where $T^*$ is the adjoint of $T$. If $T$ is $C$-symmetric for some conjugation $C$, then $T$ is called complex symmetric.
\end{dfn}
The class of complex symmetric operators includes all normal operators, Hankel operators, truncated Toeplitz operators, and Volterra integration operators. For more on complex symmetric operators and related topics, including historical comments, we refer the reader to \cite{GP2005, GP2007, GW2009, GW2010, LeeKo2016, Noor} and the references cited therein. We denote by $\mathcal{B}A(\hcl)$ the space of all bounded antilinear operators on $\hcl$. Recall that for $A \in \mathcal{B}A(\hcl)$ there exists a unique antilinear operator $A^\sharp$, known as the antilinear adjoint of $A$, defined by the equality
$$ \langle Af,g\rangle =\overline{\langle f,A^\sharp g\rangle}.$$ 
It is important to observe that  $C^\sharp=C$ for any conjugation $C$, $(AB)^{\sharp}=B^*A^{\sharp}$, and $(BA)^{\sharp}=A^{\sharp}B^*$ for $A \in \mathcal{B}A(\hcl)$, $B\in\mathcal{B}(\mathcal{H})$ (see, e.g., \cite{CAMARA1}).

\noindent Two natural conjugations $J$ and $J^\star$ on $L^2(\T)$ are defined as follows
\begin{equation}\label{conjueq}
	Jf=\bar{f}, \quad J^\star f= f^\#,\quad \text{where}\quad \bar{f}(z)=\overline{f(z)}\quad \text{and} \quad f^\#(z)=\overline{f(\bar{z})}.
\end{equation}
It is important to observe that the conjugation $J$ intertwines the operators $M_z$ and $M_{\bar{z}}$. In other words $M_z$ is $J$ -symmetric, that is $JM_zJ =M_{\bar{z}}$. Moreover, $JH^2 = \overline{H^2}$, that is $J$ maps an analytic function to a co-analytic function. On the other hand, the conjugation $J^\star$ has different nature in comparison to $J$. For example, $J^\star M_z =M_zJ^\star$ and it preserves the Hardy space $H^2$. Furthermore, it is worth mentioning that the map $J^\star$ has a connection with model spaces \cite{Cima} and Hankel operators \cite{RSN}. In this connection,   recently, C\^{a}mara-Garlicka-\L{a}nucha-Ptak raised a natural question in \cite{CAMARA2} regarding the characterization of conjugations that possess similar kinds of properties as $J$ and $J^\star$. More precisely, they characterize the class of conjugations that commute with $M_z$ and also classify the class of conjugations intertwining the operators $M_z$ and $M_{\bar{z}}$ in scaler-valued $L^2(\T)$ space \cite{CAMARA1, CAMARA2} as well as in vector-valued $L^2(\hcl)$ space \cite{CAMARA3} over the unit circle $\T$. Therefore it is natural to study the complex symmetry of $M_B$ in $L^2(\T)$, where $B$ is a finite Blaschke product. To this end, we begin with the study of the complex symmetry of $M_{z^2}$  in $L^2(\T)$ since it will help to complement the study of  $M_{B}$. It is important to observe that the commutant of $M_z$ is properly contained in the commutant of $M_{z^2}$ and also the class of conjugations intertwining the operators $M_{z^2}$, $M_{\bar{z}^2}$ in $L^2$ is bigger than the class of conjugations intertwining the operators $M_{z}$, $M_{\bar{z}}$ in $L^2$. In this direction, we begin with the following two questions.
\begin{prob}\label{Q1}
	Which class of conjugations commute with $M_{z^2}$ in $L^2(\T)$ but not with $M_z$ ? 
\end{prob}
\begin{prob}\label{Q2}
	Which class of conjugations makes the operator $M_{z^2}$ complex symmetric in $L^2(\T) $ but does not possess the same characteristic with $M_z$?
\end{prob}

Later in Section 4, we provide appropriate answers to the above questions. Note that our work mainly focus on the case for $only~M_{z^2}-commuting$ conjugations and $M_{z^2},M_{\bar{z}^2}$-intertwining (but not $M_{z},M_{\bar{z}}$-intertwining ) conjugations throughout this article. Note that the operator $M_z$ in $\C^n$-valued $L^2$ space over $\T$ is unitary equivalent to $M_B$ in the classical $L^2$ space over $\T$, where $B$ is a finite Blaschke product of degree $n$. Since the characterization of the class of conjugations that commute with $M_z$ and also the class of conjugations intertwining the operators $M_z$ and $M_{\bar{z}}$ have been studied in the vector-valued $L^2(\hcl)$ space in \cite{CAMARA3}, then in some sense such characterization results regarding $M_{z^2}$ can also be given in $L^2(\T)$ space. But the novelty of our characterization results lies in the fact that we have the precise structure of such conjugations, which can not be achieved easily using the above-mentioned unitary equivalence. Moreover, a further benefit of our description is that one can independently classify the conjugations in $L^2(\T)$, which do not possess similar properties with $M_z$. 

In other words, the precise (or major) contributions in this article are the following:
\begin{itemize}
	\item We identify the commutant of $M_B$ in $L^2$, where $B$ is a finite Blaschke product.
	\item Using the above characterization, we classify the set of conjugations that commute with $M_{z^2}$ but fail to commute with $M_z$. 
	\item Class of conjugations intertwining the operators  $M_{z^2}$ and $M_{\bar{z}^2}$ but not necessarily intertwining the operators $M_z$ and $M_{\bar{z}}$ has been identified.
	\item As a consequence of our main results, we investigate the properties of the above classes of conjugations keeping the whole Hardy space, model space, and Beurling-type subspaces invariant.
	\item At the end, we extend our discussions in connection to $M_B$, where $B$ is a finite Blaschke product.
\end{itemize}

The paper is organized as follows: In section $2$, we characterize the commutant of the multiplication operator $M_B$ in $L^2(\T)$ for any finite Blaschke product $B$. Some new conjugations in $L^2(\T)$ have been introduced as a generalization of some well-known conjugations, in section $3$. Next, we provide the answer to \textbf{Question 1} and \textbf{Question 2}
in section 4. In section 5, we discuss the properties of such conjugations that keep invariant the entire Hardy space, model space, and Beurling-type subspaces. Section 6 is devoted to studying the related results in the case of finite Blaschke product $B$ and, in particular, for $B(z)=z^n$.

\section{Commutant of $M_B$ in $L^2(\T)$}
In this section, we provide a characterization of commutants of the multiplication operator $M_B$ having symbol $B$, where $B$ is a finite Blaschke product. It is worth mentioning that, in \cite{ACZ, CoWa15, GaPa22}, the authors obtained the commutant of $M_B$ in various reproducing kernel Hilbert spaces of analytic functions. In contrast with the existing results, we obtain the commutant of $M_B$ in a non-RKHS, namely in  $L^2(\T)$, and our characterization result is different from the existing ones. We denote $\{M_B\}^\prime$ to be the set of all bounded linear operators on $L^2(\T)$ that commutes with $M_B$. In other words, if $T\in \{M_B\}^\prime$, then $T:L^2(\T)\to L^2(\T)$ is a bounded linear operator such that $TM_B=M_BT$.
Our main aim is to characterize the set $\{M_B\}^\prime$, and to achieve our goal; we need the following useful lemma.

\begin{lma}\label{lm 3.1}
	Let $\phi \in L^2(\T)$ be such that $\phi f \in L^2(\T)$  for all $f\in L^2(\T)$, then $\phi \in L^\infty (\T)$.
\end{lma}

\begin{proof}
	By hypothesis, we define a linear operator as follows:
	\begin{align*}
		A_\phi : &~ L^2(\T)\to L^2(\T)\\
		& f\to \phi f.
	\end{align*}	
	Let us denote the Graph of $A_\phi$ by $\mathcal{G}r(A_\phi)$. To serve the purpose, suppose the sequence $\{(f_n, A_\phi f_n)\} \in \mathcal{G}r(A_\phi)$  converges to $(f,g)$ in $L^2(\T)$. Then there exists a subsequence $\{f_{n_j}\}$ of $\{f_n\}$ such that $A_\phi f_{n_j} \to g$ a.e. on $\T$, in other words, $\phi f_{n_j} \to g$ a.e. on $\T$. Moreover, $f_{n_j}$ also converges to $f$ in $L^2(\T)$, and hence  there exists a subsequence $\{f_{n_{j_k}}\}$ of $\{f_{n_j}\}$ such that $f_{n_{j_k}}$ converges to $f$ a.e. on $\T$. Therefore, $\phi f_{n_{j_k}}\to \phi f$ a.e. on $\T$ as well as  $\phi f_{n_{j_k}}\to g$ a.e. on $\T$. Thus  $\phi f = g$ a.e. on $\T$ implying that $A_\phi f = g$. Therefore by Closed-Graph theorem, $A_\phi$ is a bounded operator. Next, we will show that $\phi \in L^\infty (\T)$. To proceed further,  we define
	$E_n := \{e^{i\theta} \in \T: |{\phi (e^{i\theta})}| > n \}$ for $n\in \mathbb{N}$.
	Let $\chi_{_{E_n}}$ be the characteristic function of the set $E_n$. Since $A_\phi $ is a bounded linear operator, then we get
	\begin{align*}
		\norm {A_\phi \chi_{_{E_n}}}^2 \leq \norm {\chi_{_{E_n}}}^2 \norm{A_\phi}^2. 
	\end{align*}
	If $m$ is a normalized Lebesgue measure on $\T$, then we also have
	\begin{align*}
		\norm {A_\phi \chi_{_{E_n}}}^2  = \int_{_{E_n}}|{\phi (e^{i\theta})}|^2 > n^2~ m(E_n).
	\end{align*}
	Therefore $n^2~ m(E_n) < \norm {A_\phi \chi_{_{E_n}}}^2 \leq \norm{A_\phi}^2 m(E_n) $, and hence $n> \norm {A_\phi}$ which implies $m(E_n)=0$. Consequently, we conclude that $\phi \in L^\infty (\T)$.	
\end{proof}
Before going to the main theorem in this section, we need to recall some basic facts. We know that for a finite Blaschke product $B$, the dimension of the model space $K_B:= H^2\ominus BH^2$ is finite. In fact, if the degree of $B$ is a finite natural number $n$, then $dim(K_B)=n$. Let $\{h_1,h_2,\ldots,h_n \}$ be an orthonormal basis of $K_B$. Due to \textbf{Wold-type} decomposition, the space $L^2(\T)$ can be viewed as follows
\begin{equation*}
	L^2(\T)= \bigoplus_{k=-\infty}^{\infty}K_B B^k .
\end{equation*} 
Using the above decomposition we identify $L^2(\T)$ as orthogonal direct sum of $n$-closed subspaces, that is
\begin{equation}\label{l2}
	L^2(\T)= \hcl_1(h_1,B) \oplus \hcl_2(h_2,B) \oplus \cdots \oplus \hcl_n(h_n,B),
\end{equation}
where
\begin{equation*}
	\hcl_1(h_1,B):=\bigvee_{k=-\infty}^\infty \{h_1 B^k\}, ~\hcl_2(h_2,B):=\bigvee_{k=-\infty}^\infty \{h_2 B^k\},\ldots ,\hcl_n(h_n,B):=\bigvee_{k=-\infty}^\infty \{h_nB^k\}.
\end{equation*}
Therefore for any $f\in L^2(\T) $ can be represented as $f=f_1+f_2+ \cdots +f_n$, where $f_i \in \hcl_i(h_i,B)$ for $i\in \{1,2,\cdots ,n\}$.
For a finite sequence $\{\lambda_k \}_{1\leq k \leq n}$ (maybe repeating) in $\D$, we consider the following finite Blaschke product
\begin{equation}\label{B}
	B(z)=\prod_{k=1}^n\dfrac{\lambda_k -z}{1-\bar{\lambda}_k z}~, z\in \overline{\D} .
\end{equation}
For $1\leq j\leq n$, consider the following functions
\begin{align}\label{e_j}
	e_j(z) = &\Big(\prod_{k=1}^{j-1}\dfrac{\lambda_k -z}{1-\bar{\lambda}_k z}\Big)\dfrac{\sqrt{1-|\lambda_j|^2}}{1-\bar{\lambda}_j z},  
\end{align}
then it is easy to observe that $\{e_1,e_2,\ldots ,e_n \}$ forms an orthonormal basis of $K_B$. In fact, this is a standard orthonormal basis of $K_B$, and for more details, we refer to [\cite{EMV1}, Theorem:14.7]. It is worth mentioning some nice properties of $e_j$ as follows:
\begin{enumerate}[(i)]
	\item  $e_j \in H^\infty (\T)$ for $1\leq j\leq n$.
	\item For $1\leq j\leq n$, each $e_j$ is invertible in $L^\infty (\T)$, and \begin{equation*}
		e_j ^{-1}(z)=\Big(\prod_{k=1}^{j-1}\dfrac{1-\bar{\lambda}_k z}{\lambda_k -z}\Big)\dfrac{1-\bar{\lambda}_j z}{\sqrt{1-|\lambda_j|^2}}.
	\end{equation*}
\end{enumerate}

Now we proceed to find the commutant of $M_B$. To this end, we use the decomposition of $L^2(\T)$ as an orthogonal direct sum of $n$-subspaces given in \eqref{l2}, along with the fact that $B$ is of the form \eqref{B} and $\{e_1,e_2,\ldots ,e_n \}$ is the required orthonormal basis as mentioned in \eqref{e_j}. Therefore, for any $f\in L^2(\T) \big(= \hcl_1 (e_1,B) \oplus \hcl_2(e_2,B) \oplus \cdots \oplus \hcl_n(e_n,B)\big)$ can be written as $f=f_1\oplus f_2\oplus\cdots \oplus f_n$, where $f_i \in \hcl_i(e_i,B) $ for each $i\in \{1,2,\ldots ,n \}$. Let  $S\in \{M_B \}^\prime$, then $S$ is a bounded linear operator on $L^2(\T)$ such that $M_BS=SM_B$. Since $f_j\in \hcl_j(e_j,B) $, then there exists $g_j=\sum\limits_{k=-\infty}^{\infty}a^j_kB^k \in L^2(\T)$ such that $f_j=e_jg_j$ for  $j\in \{1,2, \ldots ,n \}$. Since $e_j^{-1}\in L^\infty(\T)$, then by using the fact that $S$ commutes with $M_B$ we have
\begin{align*}
	Sf_j &= (Se_j)g_j
	=(Se_j)e_j^{-1}e_j g_j
	=(Se_j)e_j^{-1}f_j=\phi_jf_j,
\end{align*}
where $\phi_j = (Se_j)e_j^{-1} \in L^2(\T)$. Next we claim that $\phi_j \in L^\infty (\T)$ for $j\in \{1,2,\ldots ,n\}$.
We already proved that $\phi_jf_j =Sf_j\in L^2(\T)$. Now for $i\neq j$,
\begin{align*}
	\phi_jf_i & = \phi_j(e_ig_i)=e_j^{-1}e_j \phi_j(e_ig_i)=e_j^{-1}e_i(\phi_je_jg_i)=e_j^{-1}e_i S(h_j) \in L^2(\T),
\end{align*}
where $h_j= e_jg_i \in \hcl_j(e_j,B)$. Thus we have $\phi_jf=\phi_j f_1+\phi_jf_2+\cdots +\phi_jf_n\in L^2(\T)$ for any $f\in L^2(\T)$. Therefore by using  Lemma \ref{lm 3.1} we conclude $\phi_j \in L^\infty(\T)$ for $j\in \{1,2,\ldots ,n\}$. Since each $\phi_j\in L^\infty(\T)$, then for $f\in L^2(\T)$ the action of $S$ is given as follows
\begin{align*}
	Sf = & \phi_1f_1 + \phi_2f_2 + \cdots +\phi_nf_n \\
	= & M_{\phi_1}f_1 + M_{\phi_2}f_2 + \cdots + M_{\phi_n}f_n.
\end{align*} 
From now onward, we use the notation $M_{[\phi_1,\phi_2,\ldots ,\phi_n]}$ to denote the bounded linear operator on $L^2(\T)$ such that 
\begin{align}\label{mphi}
	M_{[\phi_1,\phi_2,\ldots ,\phi_n]}(f)=M_{\phi_1}f_1 + M_{\phi_2}f_2 + \cdots + M_{\phi_n}f_n.
\end{align}
Summing up, we have the following main theorem in this section, and it is extremely useful in later sections to prove other characterization results. Moreover, we expect that the following theorem will have an independent interest in the future.

\begin{thm}\label{comm}
	Let $B$ be a finite Blaschke product of degree $n$ and $S\in \{M_B\}^\prime$ in $L^2(\T)$. Then there exist $n$-number of $L^\infty(\T)$ functions $\phi_1,\phi_2,\ldots ,\phi_n$ such that $S=M_{[\phi_1,\phi_2,\ldots ,\phi_n]}$. 
\end{thm}
One can find the classical theorem for the commutant of the shift operator on $L^2(\T)$-space in \cite{RSN}(see Theorem: 2.2.5), which says that the commutant of $M_z$ in $L^2(\T)$ is $\{M_\phi:\phi \in L^\infty(\T)\}$. It is worth noticing that,  if we consider $\phi_1=\phi_2=\cdots =\phi_n=\phi$ in \eqref{mphi}, then $M_{[\phi_1,\phi_2,\ldots ,\phi_n]}$ turns out to be the classical multiplication operator $M_\phi$.  Therefore  we have the  following remark:
\begin{rmrk}\label{mzcomrem}
	Suppose $B(z)=z^2$ and $S\in \{M_{z^2}\}^\prime$. Since $\{M_z\}^\prime \subset \{M_{z^2}\}^\prime$, then
	\begin{enumerate}[(i)]
		\item for $S\in \{M_z\}^\prime$, there exists $\phi \in L^\infty$ such that $S=M_{[\phi,\phi]}=M_{\phi}$,
		\item for $S\notin \{M_z\}^\prime$, there exist $\phi_1,\phi_2 \in L^\infty$ with $\phi_1\neq\phi_2$ such that $S=M_{[\phi_1,\phi_2]}$.
	\end{enumerate}
\end{rmrk} 
In a similar fashion, one can distinguish between the commutant of $M_z$ and the commutant of $M_B$, where $B$ is a finite Blaschke product.

\section{Conjugations in $L^2(\T)$}
We begin the section by recalling two natural conjugations in $L^2(\T)$ which have been mentioned earlier, namely $J$ and $J^\star$, where $J$ is defined by $Jf=\bar{f}$ and  $J^\star$ is given by $J^\star f= f^\sharp$, where $f^\sharp(z)=\overline{f(\bar{z})}$. These two conjugations have some interesting properties. The bilateral shift  $M_z$ becomes complex symmetric in $L^2(\T)$ with respect to the  conjugation $J$, and it maps the analytic functions space  $H^2(\T)$ into the co-analytic functions space $\overline{H^2}(\T)$, that is, $J(H^2(\T))=\overline{H^2}(\T)$. On the other hand, the conjugation $J^\star$ commutes with $M_z$ and keeps the analytic function space $H^2(\T)$ invariant, that is, $J^\star (H^2(\T))\subset H^2(\T)$. As discussed earlier, we will introduce some conjugations in $L^2(\T)$ as a generalization of these two conjugations $J$ and $J^\star$ from a different point of view. Going further, for any inner function $\theta$, the natural conjugation on $K_\theta := H^2\ominus \theta H^2$ is given by
\begin{align*}
	c_\theta:& K_{\theta} \to K_{\theta}\\
	& h\mapsto \overline{zh}\theta,
\end{align*}
and it has a connection with the truncated Toeplitz operator in model spaces. For example, the truncated Toeplitz operator is $c_\theta$- symmetric (for more details, see \cite{GMR2016, GP2005}). 

For any inner function $\theta$, by Wold-type Decomposition \cite{NF2010}, we can view  $L^2(\T)$  as follows:
\begin{equation}\label{L2}
	L^2(\T) =\bigoplus _{k=-\infty}^{\infty}\theta ^k K_\theta .
\end{equation}
The proof of the above decomposition \eqref{L2} is easily obtained from the fact that the corresponding Toeplitz operator $T_\theta $ is a pure isometry and using the decomposition $L^2(\T)=H^2 \oplus \overline{H^2_0}$. For reader convenience, we briefly discuss the proof of the above decomposition. Applying the classical \textbf{Von Neumann-Wold Decomposition} theorem on $H^2$ associated with the shift operator $T_\theta$ we obtain the following:
$H^2 (\T) =\bigoplus\limits _{k=0}^{\infty}\theta ^k K_\theta .$
Therefore, for any $h\in H^2(\T)$, there exists a sequence $\{h_k\}_{k=0}^\infty$ in $K_\theta$ such that $h=\sum\limits_{k=0}^{\infty}h_k\theta^k$ and $\norm{h}^2=\sum\limits_{k=0}^{\infty}\norm{h_k}^2$. Note that, $\theta^k K_\theta \perp \theta^m K_\theta$ for any $k\neq m\in \Z$. As a consequence, we have
$\bigoplus\limits _{k=-\infty}^{\infty}\theta ^k K_\theta \subseteq L^2(\T).$
Let  $f\in L^2(\T)=H^2\oplus \overline{H^2_0}$, then there exist $f_1\in H^2$ and $f_2 \in \overline{H^2_0}$ such that $f=f_1\oplus f_2$. Again, $f_1=\sum\limits_{k=0}^{\infty}\tilde{f}_k\theta^k$, where $\{\tilde{f}_k\}_{k=0}^\infty \in K_\theta$ and $f_2=\overline{zg}$ for some $g\in H^2$. Moreover, we have $g=\sum\limits_{k=0}^{\infty}g_k\theta^k$, where $\{g_k\}_{k=0}^\infty \in K_\theta$ and hence 
\begin{align*}
	f_2  = \sum_{k=0}^{\infty}\overline{zg_k}{\theta}^{-k} =\sum_{k=0}^{\infty}\overline{zg_k}\theta \theta^{-(k+1)}=\sum_{k=0}^{\infty}(c_\theta g_k)\theta^{-(k+1)} 
	&= \sum_{k=-\infty}^{-1}\tilde{g}_k\theta^k, 
\end{align*}
where $\tilde{g}_k=c_\theta g_k \in K_\theta $.  Finally, combining all the above facts, we conclude
$f=\sum\limits_{k=-\infty}^{\infty}f_k\theta^k,$
where $f_k= \tilde{f}_k$ for $k\geq 0$ and $f_k= \tilde{g}_k$ for $k<0$.

Next, we define conjugations extending the natural conjugation $c_\theta$ from the model space $K_\theta$ corresponding to an inner function $\theta$  to the entire Hardy space in two different ways.
First, we define $\mathscr{C}_\theta$ on $L^2(\T)$	as follows:
\begin{align*}
	\mathscr{C}_\theta: & L^2(\T) \to L^2(\T)\\
	& f \mapsto \sum_{k=-\infty}^{\infty}(c_\theta f_k) \theta ^{-k}, \quad \text{where}~ f=\sum_{k=-\infty}^{\infty}f_k\theta ^{k} ,f_k \in K_\theta, ~\text{that is}
\end{align*}
\begin{equation}\label{C1} \mathscr{C}_\theta(\sum_{k=-\infty}^{\infty}f_k\theta ^k )= \sum_{k=-\infty}^{\infty}\overline{zf_k}\theta ^{-k+1}.
\end{equation}
It is easy to check that  $\mathscr{C}_\theta$ is a conjugation on $L^2(\T)$. In particular, if  $\theta (z)=z$, then the action of $\mathscr{C}_z$ is given by
\begin{equation}\label{c_z}
	\mathscr{C}_z(\sum_{k=-\infty}^{\infty}a_kz^k)=\sum_{k=-\infty}^{\infty}\bar{a}_k\bar{z}^k.
\end{equation}
Note that the action of $\mathscr{C}_\theta$ given in \eqref{C1} can also be rewritten as
\begin{equation}\label{conuzsq}
	\mathscr{C}_\theta (f)=\theta \overline{zf},
\end{equation} 
and hence $\mathscr{C}_z(f)=\bar{f}$. In other words, $\mathscr{C}_z=J$, and we can think $\mathscr{C}_\theta $ is a generalization of the conjugation $J$. These types of conjugations \eqref{conuzsq} have already been introduced in \cite{CAMARA2}, where they use it to characterize all  $M_z$-commuting and $M_z$, $M_{\bar{z}}$- intertwining conjugations preserving model spaces.

Now we define another conjugation (different from $\mathscr{C}_\theta$ above) corresponding to an inner function $\theta$, namely $\mathscr{C}_\theta ^\star :L^2(\T)\to L^2(\T)$ such that if $f=\sum\limits_{k=-\infty}^{\infty}f_k\theta ^{k}$ with $f_k \in K_\theta$, then 
$\mathscr{C}_\theta^\star(f)=\sum\limits_{k=-\infty}^{\infty}(c_\theta f_k) \theta ^{k},$
which implies that
\begin{equation}\label{C*}
	\mathscr{C}_\theta^\star(\sum_{k=-\infty}^{\infty}f_k\theta ^k )= \sum_{k=-\infty}^{\infty}\overline{zf_k}\theta ^{k-1}.
\end{equation}
In particular, for $\theta(z)=z$, we have $\mathscr{C}_z ^\star=J^\star$. These types of conjugations can also be found in \cite{CDPS}. 

In the beginning, we have already discussed certain properties of two natural conjugations $J$ and $J^\star$. It is easy to observe that those properties also hold for $\mathscr{C}_\theta$ and $\mathscr{C}_\theta ^\star$. To be more specific, $\mathscr{C}_\theta M_\theta =M_{\bar{\theta}}\mathscr{C}_\theta $, $\mathscr{C}_\theta (H^2)=\overline{H^2}$, $\mathscr{C}_\theta^\star M_\theta =M_\theta \mathscr{C}_\theta^\star$, and $\mathscr{C}_\theta^\star (H^2)=H^2$. Moreover, the conjugation $\mathscr{C}_{z^2}^\star$ has one more extra property, namely $\mathscr{C}_{z^2}^\star M_z \neq M_z \mathscr{C}_{z^2}^\star$ and, in general,  $\mathscr{C}_\theta^\star M_z \neq M_z \mathscr{C}_\theta^\star$ for an inner function $\theta$.

\section{Conjugations Intertwining $M_{z^2}$, $M_{\bar{z}^2}$ and Commuting with $M_{z^2}$ In $L^2$}
The definition of the conjugations $J$, $J^\star$ in $L^2$ (see \eqref{conjueq}) and some of their properties have already been introduced in Section 1. In \cite{CAMARA1,CAMARA2}, the authors obtained the characterization of all conjugations $C$ intertwining the operators $M_z$, $M_{\bar{z}}$ in terms  $J$ in $L^2$ (see \cite[Theorem 2.2]{CAMARA2}).  Moreover, they also characterized all conjugations $C$ that commute with the operator $M_z$ in terms of $J^\star$ in $L^2$ (see \cite[Theorem 2.4]{CAMARA2}). 
Motivated by the characterization results mentioned above, we raise \textbf{Question~\ref{Q1}} and \textbf{Question~\ref{Q2}} in Section 1. Our main aim in this section is to provide appropriate answers of these questions. First, we characterize all conjugations $C$ that satisfies the identity $CM_{z^2}=M_{\bar{z}^2}C$. One of the importance of our Theorem~\ref{th4.1} below is that the class of conjugations $C$ intertwining the operators $M_{z^2}$, $M_{\bar{z}^2}$ in $L^2$ is bigger than the class  of conjugations $C$ intertwining the operators $M_{z}$, $M_{\bar{z}}$ in $L^2$.  Before going to the main theorems, we define $$\phi^e(z)=\dfrac{\phi(z)+\phi(-z)}{2} \quad \text{and}\quad  \phi^o(z)=\dfrac{\phi(z)-\phi(-z)}{2} \quad \text{for}\quad  \phi \in L^2(\T).$$ Note that, $\phi^e \in \bigvee_{k=-\infty}^{\infty}\{z^{2k}\}=\mathcal{H}_1(1,z^2)$ and $\phi^o \in \bigvee_{k=-\infty}^{\infty}\{z^{2k+1}\}=\mathcal{H}_2(z,z^2)$. 
Next, we state and prove one of our main results in this section which involve the conjugation $\mathscr{C}_{z^2}(f)=z\overline{f}$ for $f\in L^2$ (see \eqref{conuzsq}).
\begin{thm}\label{th4.1}
	Let $C:L^2(\T)\to L^2(\T)$ be an arbitrary conjugation  such that $CM_{z^2}=M_{\bar{z}^2}C$, then there exist $\phi_1,\phi_2\in L^\infty$ such that $C= M_{[\phi_1,\phi_2]}\mathscr{C}_{z^2}$, where $\phi_1,\phi_2$ satisfies $|\phi_1|^2+|\phi_2|^2=2$ and $\phi_1^e= \phi_2^e$. Conversely, any conjugation $C$ of the form $M_{[\phi_1,\phi_2]}\mathscr{C}_{z^2}$ for $\phi_1, \phi_2 \in L^\infty$ satisfies $CM_{z^2}=M_{\bar{z}^2}C$.
\end{thm} 
\begin{proof}
	Suppose $C$ is a conjugation on $L^2(\T)$ such that $CM_{z^2}=M_{\bar{z}^2}C$. Then, it is easy to observe that the unitary operator $C\circ\mathscr{C}_{z^2}$ commutes with $M_{z^2}$. Therefore by applying Theorem \ref{comm}, we have 
	\begin{equation}\label{symuniop}
		C\circ\mathscr{C}_{z^2}=  M_{[\phi_1,\phi_2]} \quad \text{for}\quad  \phi_1,\phi_2 \in L^\infty(\T)
	\end{equation} 
	and hence $C = M_{[\phi_1,\phi_2]} \mathscr{C}_{z^2} $. By considering the decomposition $L^2(\T)=\hcl_1(1,z^2) \oplus \hcl_2(z,z^2)$, we have the following block-matrix representation of $M_{[\phi_1,\phi_2]}$
	
	\begin{equation*}
		M_{[\phi_1,\phi_2]}= 
		\begin{bmatrix}  
			M_{\phi_1^e} && M_{\phi_2^o}\\
			\\ 
			M_{\phi_1^o} && M_{\phi_2^e}\\
		\end{bmatrix} \text{~and, ~hence~}
		M_{[\phi_1,\phi_2]}^*= 
		\begin{bmatrix}  
			M_{\bar{\phi_1^e}} && M_{\bar{\phi_1^o}}\\
			\\ 
			M_{\bar{\phi_2^0}} && M_{\bar{\phi_2^e}}\\
		\end{bmatrix}.  
	\end{equation*}
	From \eqref{symuniop} it follows that $M_{[\phi_1,\phi_2]}$ is a unitary operator, and therefore we obtain the following conditions on the symbols.
	\begin{equation}\label{eq-im}
		\begin{cases*}
			|\phi_1^e|^2+|\phi_1^o|^2=1 ~,~ |\phi_1^e|^2+|\phi_1^o|^2=1,\\
			|\phi_1^e|^2+|\phi_2^o|^2=1 ~,~ |\phi_1^o|^2 + |\phi_2^e|^2=1,\\
			\bar{\phi_1^e}\phi_2^o+\bar{\phi_1^o}\phi_2^e=0 ~,~ \bar{\phi_1^o}\phi_1^e+\bar{\phi_2^e}\phi_2^o=0,
		\end{cases*}
	\end{equation}
	which by simplifying, reduces to 
	\begin{equation}\label{cond4.1}
		\begin{cases*}
			|\phi_1|^2+|\phi_2|^2=2\\
			|\phi_1^o+\phi_2^e|^2+|\phi_1^e+\phi_2^o|^2=2.
		\end{cases*}
	\end{equation}
	Moreover, due to the involutive property of the conjugation $C$, we have
	\begin{equation*}
		M_{[\phi_1,\phi_2]} \mathscr{C}_{z^2}M_{[\phi_1,\phi_2]} \mathscr{C}_{z^2}  =  I\implies \mathscr{C}_{z^2}M_{[\phi_1,\phi_2]} \mathscr{C}_{z^2}   =   M_{[\phi_1,\phi_2]}^*,
	\end{equation*}
	which by applying on the basis elements $1$ or $z$ yields $\phi_1^e= \phi_2^e$. Therefore, the conditions in \eqref{cond4.1} merged to $|\phi_1|^2+|\phi_2|^2=2$. 
	
	For the converse, suppose $C=M_{[\phi_1,\phi_2]}\mathscr{C}_{z^2}$ for $\phi_1, \phi_2 \in L^\infty$. Then for $f\in L^2(\T)$, there exist $f^e \in \hcl_1(1,z^2)$ and $f^o \in \hcl_2(z,z^2)$ such that $f=f^e\oplus f^o$. Since $\mathscr{C}_{z^2}(f)= z\overline{f}$, then we have
	\begin{equation*}
		CM_{z^2}(f)=C(z^2f^e\oplus z^2f^o)= M_{[\phi_1,\phi_2]}\mathscr{C}_{z^2}(z^2f^e\oplus z^2f^o) =M_{[\phi_1,\phi_2]}(\overline{zf^o}\oplus\overline{zf^e})=\phi_1\overline{zf^o}+\phi_2\overline{zf^e}.
	\end{equation*}
	Similarly, we obtain $M_{\bar{z}^2}C(f)=M_{\bar{z}^2}M_{[\phi_1,\phi_2]}(z\overline{f^o}\oplus z\overline{f^e}) = \phi_1\overline{zf^o}+\phi_2\overline{zf^e}$. This completes the proof.
\end{proof}
As a consequence of the above theorem, we characterize all conjugations $C$ intertwining the operators $M_{z^2}$, $M_{\bar{z}^2}$ in $L^2$ in terms of $J$ as follows.
\begin{crlre}\label{cr4.1}
	Let $C$ be a conjugation on $L^2$ such that $CM_{z^2}=M_{\bar{z}^2}C$, then $C=M_{[\psi_1,\psi_2]}J$ for  $\psi_1,\psi_2\in L^\infty$ with $\psi_1^o= \psi_2^o$ and $|\psi_1|^2+|\psi_2|^2=2$. Conversely, any conjugation $C$ of the form $M_{[\psi_1,\psi_2]}J$ for $\psi_1,\psi_2 \in L^\infty$ satisfies $CM_{z^2}=M_{\bar{z}^2}C$.
\end{crlre}
\begin{proof}
	By definition of $\mathscr{C}_{z^2}$ it follows that $\mathscr{C}_{z^2} = M_z J$.  Therefore by using Theorem \ref{th4.1}, we conclude $C=M_{[\phi_1,\phi_2]} \mathscr{C}_{z^2}$, where $\phi_1^e= \phi_2^e$ and $|\phi_1|^2+|\phi_2|^2=2$. Note that $M_{[\phi_1,\phi_2]}M_z= M_{[z\phi_2,z\phi_1]}$ and hence $C=M_{[\psi_1,\psi_2]}J$, where $\psi_1 (=z\phi_2) , \psi_2 (= z\phi_1)\in L^\infty$. Moreover, the involutive property of the conjugation $C$ implies $\psi_1^o= \psi_2^o$, and
	since $M_{[\psi_1,\psi_2]}$ is unitary, then we have  $|\psi_1|^2+|\psi_2|^2=2$. The converse is straightforward.
\end{proof}

\begin{rmrk}\label{rem-4.5}
	Note that, the  conditions \eqref{eq-im} in Theorem \ref{th4.1} after substituting $\mathscr{C}_{z^2} = M_z J$ and $\psi_1 =z\phi_2 , \psi_2= z\phi_1$, change to 
	\begin{equation}\label{im-co}
		\begin{cases*}
			|\psi_2^o|^2+|\psi_2^e|^2=1 ~,~ |\psi_2^o|^2+|\psi_2^e|^2=1,\\
			|\psi_2^o|^2+|\psi_1^e|^2=1 ~,~ |\psi_2^e|^2 + |\psi_1^o|^2=1,\\
			\bar{\psi_2^o}\psi_1^e+\bar{\psi_2^e}\psi_1^o=0 ~,~ \bar{\psi_2^e}\psi_2^o+\bar{\psi_1^o}\psi_1^e=0.
		\end{cases*}
	\end{equation}	
	It should be noted that if we combine the above set of conditions, it reduces to the single condition $|\psi_1|^2+|\psi_2|^2=2$ obtained in Corollary \ref{cr4.1}. The usefulness of the conditions in \eqref{im-co} will be realized in a later section.  
	
\end{rmrk}

\begin{rmrk}\label{rem-IMP}
	It is important to note that if we assume $\psi_1 =\psi_2$ in Corollary\eqref{cr4.1}, then $C=M_{\psi_1}J$ intertwining the operators $M_z$, $M_{\bar{z}}$ and hence intertwining the operators $M_{z^2}$, $M_{\bar{z}^2}$ as well in $L^2$. On the other hand, if $\psi_1 \neq \psi_2$, then $C=M_{[\psi_1,\psi_2]}J$ intertwining the operators $M_{z^2}$, $M_{\bar{z}^2}$ but not intertwining the operators $M_z$ and $M_{\bar{z}}$ in $L^2$.
\end{rmrk}

According to the previous discussion, our next aim is to characterize all the conjugations commuting with $M_{z^2}$. For instance, the conjugation $\mathscr{C}_{z^2}^\star$ (see \eqref{C*}) commutes with $M_{z^2}$. To be more specific, for $f(z)=\sum\limits_{k=-\infty}^{\infty}a_nz^n\in L^2$, the action of $\mathscr{C}_{z^2}^\star$ is given by 
\begin{align*}
	\mathscr{C}_{z^2}^\star(\sum_{k=-\infty}^{\infty}a_nz^n)&=\sum_{k=-\infty}^{\infty}\bar{a}_{2n}z^{2n+1} + \sum_{k=-\infty}^{\infty}\bar{a}_{2n+1}z^{2n}~\implies~
	\mathscr{C}_{z^2}^\star(f(z))  = z\overline{f^e(\bar{z})} +\bar{z}\overline{f^o(\bar{z})}.
\end{align*}
Now we have the following theorem concerning $M_{z^2}$-commuting conjugations.
\begin{thm}\label{th4.2}
	If $C$ is a conjugation on $L^2(\T)$ such that $CM_{z^2}=M_{z^2}C$, then there exist $\zeta_1,\zeta_2 \in L^\infty$ satisfying $|\zeta_1|^2+|\zeta_2|^2=2,
	|\bar{z}\zeta_1^o+z\zeta_2^e|^2+|\bar{z}\zeta_1^e+z\zeta_2^o|^2=2, \zeta_2^e(\bar{z})=\zeta_1^e(z), \zeta_2^o(\bar{z})=z^2\zeta_2^0(z), z^2\zeta_1^o(\bar{z})=\zeta_1^o(z)$ such that
	$$C=M_{[\zeta_1,\zeta_2]}\mathscr{C}_{z^2}^\star.$$
	Conversely, any conjugation $C$ of the form $M_{[\zeta_1,\zeta_2]}\mathscr{C}_{z^2}^\star$ for  $\zeta_1,\zeta_2 \in L^\infty$ commutes with $M_{z^2}$.	
\end{thm}
\begin{proof}
	Let $C$ be a conjugation on $L^2(\T)$ such that $CM_{z^2}=M_{z^2}C$. Since $\mathscr{C}_{z^2}^\star M_{z^2}=M_{z^2}\mathscr{C}_{z^2}^\star$, then the unitary operator $C\circ \mathscr{C}_{z^2}^\star$ also commutes with $M_{z^2}$ in $L^2(\T)$ . Therefore by applying Theorem \ref{comm}, there exist $\zeta_1,\zeta_2\in L^\infty(\T)$ such that $C\circ \mathscr{C}_{z^2}^\star = M_{[\zeta_1,\zeta_2]}$.  Now by considering the decomposition $L^2(\T)=\hcl_1(1,z^2) \oplus \hcl_2(z,z^2)$, we have the following block-matrix representation of $M_{[\zeta_1,\zeta_2]}$ 
	\begin{equation*}
		M_{[\zeta_1,\zeta_2]}= 
		\begin{bmatrix}  
			M_{\zeta_1^e} && M_{\zeta_2^o}\\
			\\ 
			M_{\zeta_1^o} && M_{\zeta_2^e}\\
		\end{bmatrix} \text{~and, ~hence~}
		M_{[\zeta_1,\zeta_2]}^*= 
		\begin{bmatrix}  
			M_{\bar{\zeta_1^e}} && M_{\bar{\zeta_1^o}}\\
			\\ 
			M_{\bar{\zeta_2^0}} && M_{\bar{\zeta_2^e}}\\
		\end{bmatrix}.  
	\end{equation*}
	Using the fact that $M_{[\zeta_1,\zeta_2]}$ is a unitary operator we obtain the following conditions
	\begin{equation}\label{Im-cond}
		\begin{cases*}
			|\zeta_1^e|^2+|\zeta_1^o|^2=1 ~,~ |\zeta_1^e|^2+|\zeta_1^o|^2=1,\\
			|\zeta_1^e|^2+|\zeta_2^o|^2=1 ~,~ |\zeta_1^o|^2 + |\zeta_2^e|^2=1,\\
			\bar{\zeta_1^e}\zeta_2^o+\bar{\zeta_1^o}\zeta_2^e=0 ~,~ \bar{\zeta_1^o}\zeta_1^e+\bar{\zeta_2^e}\zeta_2^o=0,
		\end{cases*}
	\end{equation}	
	which by simplifying, becomes
	\begin{equation*}
		\begin{cases*}
			|\zeta_1|^2+|\zeta_2|^2=2,\\
			|\bar{z}\zeta_1^o+z\zeta_2^e|^2+|\bar{z}\zeta_1^e+z\zeta_2^o|^2=2.
		\end{cases*}
	\end{equation*}
	Moreover, the involutive property of the conjugation $C$ implies  
	\begin{equation*}
		\mathscr{C}_{z^2}^\star M_{[\zeta_1,\zeta_2]} \mathscr{C}_{z^2}^\star  =   M_{[\zeta_1,\zeta_2]}^*,
	\end{equation*}
	which by applying on the basis vectors $1$ and $z$, we get
	$\zeta_2^e(\bar{z})=\zeta_1^e(z)$, $\zeta_2^o(\bar{z})=z^2\zeta_2^0(z)$  and 
	$z^2\zeta_1^o(\bar{z})=\zeta_1^o(z)$, $\zeta_1^e(\bar{z}) =\zeta_2^e(z)$ respectively.
	
	Conversely, if $C=M_{[\zeta_1,\zeta_2]}\mathscr{C}_{z^2}^\star$ for  $\zeta_1,\zeta_2 \in L^\infty$ , then it is easy to verify that $CM_{z^2}=M_{z^2}C$. This completes the proof.
\end{proof}

\begin{rmrk}
	It is interesting to note that the conjugation $\mathscr{C}_{z^2}^\star$ commutes with $M_{z^2}$ but  not with $M_z$. But in the above Theorem~\ref{th4.2}, we characterize $M_{z^2}$-commuting conjugations in terms of $M_{[\zeta_1,\zeta_2]}$ and $\mathscr{C}_{z^2}^\star$, where the composition operator $M_{[\zeta_1,\zeta_2]}\mathscr{C}_{z^2}^\star$ may commutes with $M_z$. 
\end{rmrk}	
In general, it is quite difficult to distinguish between the class of conjugations commuting with $M_z$ and the class of conjugations not commuting with $M_z$ inside $M_{z^2}$-commuting conjugations. To get a better understanding of the same,  we revisit the study of $M_{z^2}$-commuting conjugations in the following corollary.

\begin{crlre}\label{cr4.2}
	Any $ M_{z^2}$-commuting  conjugation  $C$ on $L^2$, there exist two $L^\infty$-functions $\xi_1,\xi_2$ on $\T$ such that $C=M_{[\xi_1,\xi_2]}J^\star$, where $\xi_1$ and $\xi_2$ satisfy $|\xi_1|^2+|\xi_2|^2=2,|\xi_2^e+\xi_1^o|^2+| \xi_2^o+\xi_1^e|^2=2,\xi_1^o(\bar{z})=\xi_2^o(z), \xi_1^e(\bar{z})=\xi_1^e(z), \text{and}~ \xi_2^e(\bar{z})=\xi_2^e(z)$. Conversely, any conjugation $C$ of the form $M_{[\xi_1,\xi_2]}J^\star$ for  $\xi_1,\xi_2 \in L^\infty $ commutes with $M_{z^2}$.
\end{crlre}
\begin{proof}
	Let $C:L^2(\T)\to L^2(\T)$ be a conjugation  such that $CM_{z^2}=M_{z^2}C$. Then by using the above Theorem \ref{th4.2}, we get two $L^\infty$-functions $\zeta_1,\zeta_2$ such that $C= M_{[\zeta_1,\zeta_2]}\mathscr{C}_{z^2}^\star$, and
	\begin{equation}\label{eq4.2}
		\begin{cases}
			|\zeta_1|^2+|\zeta_2|^2=2,
			|\bar{z}\zeta_1^o+z\zeta_2^e|^2+|\bar{z}\zeta_1^e+z\zeta_2^o|^2=2,\\
			\zeta_2^e(\bar{z})=\zeta_1^e(z), \zeta_2^o(\bar{z})=z^2\zeta_2^0(z), z^2\zeta_1^o(\bar{z})=\zeta_1^o(z).
		\end{cases}
	\end{equation} 
	Moreover,  for any $f(=f^e\oplus f^o)\in L^2$ and $z\in \T$, we have $$(\mathscr{C}_{z^2}^\star f) (z)= z\overline{f^e(\bar{z})}+\bar{z}\overline{f^o(\bar{z})}=\bar{z}\overline{f^o(\bar{z})} \oplus z\overline{f^e(\bar{z})} =(M_{[z,\bar{z}]}J^\star f)(z) ,$$
	and hence $\mathscr{C}_{z^2}^\star = M_{[z,\bar{z}]}J^\star $ which immediately implies
	\begin{equation}\label{eq4.3}
		C=M_{[\zeta_1,\zeta_2]}M_{[z,\bar{z}]}J^\star=M_{[\xi_1,\xi_2]}J^\star,
	\end{equation}
	where $\xi_1=z\zeta_2$  and $\xi_2=\bar{z}\zeta_1$. Next by identifying the relation between $\{\zeta_1,\zeta_2\}$ and $\{\xi_1,\xi_2\}$ we obtain
	\begin{equation}\label{eq4.4}
		\begin{cases}
			\zeta_1^e = z\xi_2^o ,~ \zeta_1^o = z\xi_2^e,\\
			\zeta_2^e = \bar{z}\xi_1^o ,~ \zeta_2^o = \bar{z}\xi_1^e.
		\end{cases}
	\end{equation}
	Thus the conditions in \eqref{eq4.2} becomes
	\begin{equation}\label{eq4.5}
		\begin{cases}
			|\xi_1|^2+|\xi_2|^2=2,  |\xi_2^e+\xi_1^o|^2+| \xi_2^o+\xi_1^e|^2=2, \\
			z\xi_1^o(\bar{z})=z\xi_2^o(z), z\xi_1^e(\bar{z})=z^2\bar{z}\xi_1^e(z), z^2\bar{z}\xi_2^e(\bar{z})=z\xi_2^e(z),
		\end{cases}
	\end{equation}
	which by simple modification reduces to 
	\begin{equation}\label{cond4.11}
		\begin{cases}
			|\xi_1|^2+|\xi_2|^2=2,|\xi_2^e+\xi_1^o|^2+| \xi_2^o+\xi_1^e|^2=2, \\
			\xi_1^o(\bar{z})=\xi_2^o(z), \xi_1^e(\bar{z})=\xi_1^e(z), \xi_2^e(\bar{z})=\xi_2^e(z).
		\end{cases}
	\end{equation}
	For the converse, it is easy to verify that for any conjugation $C$ on $L^2$ having the form $C=M_{[\xi_1,\xi_2]}J^\star$ commutes with $M_{z^2}$. This completes the proof.  
\end{proof}

\begin{rmrk}\label{rem-1}
	It is important to observe that the conditions $|\xi_1|^2+|\xi_2|^2=2$ and $|\xi_2^e+\xi_1^o|^2+| \xi_2^o+\xi_1^e|^2=2 $ in \eqref{cond4.11} of the above Corollary \ref{cr4.2} can be reduced from the original conditions \eqref{Im-cond} of Theorem \ref{th4.2}. Going further, one should always remember that the original explicit conditions behind the above two reduced conditions are
	\begin{equation}\label{Im-cond2}
		\begin{cases*}
			|\xi_1^e|^2+|\xi_1^o|^2=1 ~,~ |\xi_1^e|^2+|\xi_1^o|^2=1,\\
			|\xi_1^e|^2+|\xi_2^o|^2=1 ~,~ |\xi_1^o|^2 + |\xi_2^e|^2=1,\\
			\bar{\xi_2^o}\xi_1^e+\bar{\xi_2^e}\xi_1^o=0 ~,~ \bar{\xi_2^e}\xi_2^o+\bar{\xi_1^o}\xi_1^e=0.
		\end{cases*}
	\end{equation}	 
\end{rmrk}

Now we are in a position to distinguish between the class of conjugations commuting with $M_z$ and the class of conjugations not commuting with $M_z$ inside $M_{z^2}$-commuting conjugations with the help of the above Corollary \ref{cr4.2}.
\begin{itemize}
	\item If the two $L^\infty$- functions $\xi_1$ and $\xi_2$ appearing in the representation of $M_{z^2}$-commuting conjugation $C$ as described in \ref{cr4.2} becomes identical, that is, if $\xi_1=\xi_2$, then the conclusion can be drawn that $C$ actually commutes with $M_z$. Conversely, if  $C$ commutes with $M_z$, then by using Corollary \ref{cr4.2} and Remark \ref{mzcomrem}, we conclude  $C=M_{[\xi,\xi]}J^\star$, where $\xi \in  L^\infty $ satisfying $|\xi|^2+|\xi|^2=2,\xi^o(\bar{z})=\xi^o(z), \xi^e(\bar{z})=\xi^e(z), $ and $\xi^e(\bar{z})=\xi^e(z)$ which implies $C=M_{\xi}J^\star$, where $\xi \in  L^\infty $ such that $|\xi|=1$ and $\xi(z)=\xi(\bar{z})$ and it is consistent with the condition obtained in \cite{CAMARA2}.
	
	\item As a matter of fact, if $\xi_1 \neq \xi_2$, then $C=M_{[\xi_1,\xi_2]}J^\star$ commutes with $M_{z^2}$ but not with $M_z$. Conversely, if $C$ commutes only with $M_{z^2}$, then there exist $\xi_1 \neq \xi_2$ satisfying
	\begin{equation*}
		\hspace{0.5in} |\xi_1|^2+|\xi_2|^2=2,|\xi_2^e+\xi_1^o|^2+| \xi_2^o+\xi_1^e|^2=2,\\
		\xi_1^o(\bar{z})=\xi_2^o(z), \xi_1^e(\bar{z})=\xi_1^e(z), ~\text{and}~ \xi_2^e(\bar{z})=\xi_2^e(z)
	\end{equation*}  
	such that $C=M_{[\xi_1,\xi_2]}J^\star$ and we call such type of conjugations as $only~M_{z^2}-commuting$ conjugations.
\end{itemize}
Therefore, the above discussions completely describe the class of conjugations that commute with $M_{z^2}$ in  $L^2$-space.	
\section{The class of Conjugations that keep Invariant Various subspaces of $L^2(\T)$}	
Our main aim in this section is to study the behavior of two special classes of conjugations, namely $M_{z^2},M_{\bar{z}^2}$-intertwining, and $M_{z^2}$-commuting conjugations that keep invariant various subspaces of $L^2$ like the Hardy space $H^2$, model spaces and shift-invariant subspaces. The study of $M_{z},M_{\bar{z}}$-intertwining, and $M_{z}$-commuting conjugations preserving  $H^2$, model spaces, and shift-invariant subspaces have been done recently in \cite{CAMARA2}. The significant difference between our study and the study in \cite{CAMARA2} is that we characterize  $only~M_{z^2}-commuting$ and $M_{z^2},M_{\bar{z}^2}$-intertwining  (but not $M_{z},M_{\bar{z}}$-intertwining ) conjugations preserving the above mentioned subspaces. 
\subsection{Hardy Space Preserving Conjugations}
$~$
\par We have already identified the class of conjugations that intertwine the operators $M_{z^2}$, $M_{\bar{z}^2}$ and also the class of conjugations that commute with $M_{z^2}$ in terms of known conjugations in the previous section. Note that the involutive property of a conjugation guarantees that if $H^2$ is invariant under any conjugation $C$ in $L^2$, then $C(H^2)=H^2$. We begin with the following interesting result.
\begin{thm}\label{Th5.1}
	In $L^2$, there is no conjugation that intertwines  the operators $M_{z^2},M_{\bar{z}^2}$ and keeps the Hardy space $H^2$ invariant.
\end{thm}

\begin{proof}
	Let  $C:L^2 \to L^2$ be a conjugation such that $CM_{z^2}=M_{\bar{z}^2}C$, then by Corollary \ref{cr4.1} we get $\psi_1,\psi_2 \in L^\infty$ such that $C=M_{[\psi_1,\psi_2]}J$ with $\psi_1^o= \psi_2^o$ and $|\psi_1|^2+|\psi_2|^2=2$. Now if we assume that $C(H^2)\subset H^2$, then 
	\vspace{0.1in}
	
	\begin{enumerate}[(i)]
		\item $C(1)\in H^2 \implies M_{[\psi_1,\psi_2]}J (1) \in H^2 
		\implies \psi_1 \in H^2,~ \text{and hence} ~\psi_1\in H^\infty .$
		\vspace{0.1in}
		
		\item $C(z)\in H^2 \implies 
		M_{[\psi_1,\psi_2]}J(z)  \in H^2 
		\implies \bar{z}\psi_2 \in H^2,~ \text{and hence} ~  \psi_2 \in H^\infty .$
	\end{enumerate} 	
	Now if $n=2k+1,k\geq 0$, then we obtain
	\begin{equation*}
		\langle C(z^{n+1}),\bar{z}\rangle = \langle M_{[\psi_1,\psi_2]}J(z^{n+1}),\bar{z}\rangle = \langle \psi_1\bar{z}^{n+1},\bar{z} \rangle = \langle \psi_1, z^n \rangle.
	\end{equation*}
	Since $C(H^2)\subset H^2$, then we conclude 
	$\langle C(z^{n+1}),\bar{z}\rangle =0$, and hence $\langle \psi_1, z^n \rangle =0$. In other words, we obtain $\psi_1^o=0$. 
	Similarly,
	$\langle C(z^{n+1}), \bar{z}^2\rangle = \langle \psi_1, z^{n-1} \rangle =0$,
	which implies $\psi_1^e = 0$ and hence $\psi_1 =0$.
	Moreover, for $n=2k, k\geq 0$, we have
	\begin{equation*}
		\langle C(z^{n+1}),\bar{z}\rangle = \langle M_{[\psi_1,\psi_2]}J(z^{n+1}),\bar{z}\rangle = \langle \psi_2\bar{z}^{n+1},\bar{z} \rangle = \langle \psi_2, z^n \rangle =0.
	\end{equation*}
	Therefore, $\psi_2^e =0$ and since $\psi_2^o =\psi_1^o$  we get $\psi_2=0$, which contradicts the fact $|\psi_1|^2+|\psi_2|^2=2.$
\end{proof}
Next, we discuss the class of $only~M_{z^2}-commuting$  conjugations that keep the Hardy space invariant.  

\begin{thm}
	Let $C$ be a \textquotedblleft $only~M_{z^2}-commuting"$ conjugation on $L^2$ such that $C(H^2)\subset H^2$, then $C=M_{[\xi_1,\xi_2]}J^\star$, where
	\begin{enumerate}[(i)]
		\item $\xi_1=a_0+a_1z$, $\xi_2 =a_1\bar{z}+b_0$ for some $a_0,a_1,b_0 \in \C, $ and
		\item $|a_0|^2+|a_1|^2=1,|b_0|^2+|a_1|^2=1,a_0\bar{a_1}+a_1\bar{b}_0=0.$
	\end{enumerate}
\end{thm}
\begin{proof}
	Let $C$ be a $only~M_{z^2}-commuting$ conjugation on $L^2$. Then by Corollary \ref{cr4.2}, there exist  $\xi_1,\xi_2\in L^\infty$ such that $C=M_{[\xi_1,\xi_2]}J^\star $, where $|\xi_1|^2+|\xi_2|^2=2,|\xi_2^e+\xi_1^o|^2+| \xi_2^o+\xi_1^e|^2=2,\xi_1^o(\bar{z})=\xi_2^o(z), \xi_1^e(\bar{z})=\xi_1^e(z), \xi_2^e(\bar{z})=\xi_2^e(z)$. Moreover, $C(1)\in H^2$ and $C(z)\in H^2$ implies $\xi_1\in H^\infty$ and $z\xi_2 \in H^\infty$ respectively. Suppose $\xi_1=\sum\limits_{k=0}^{\infty}a_kz^k $ and $\xi_2 =\sum\limits_{k=-1}^{\infty}b_kz^k$.
	Now, combining the facts $\xi_1^e\in H^\infty$ and $\xi_1^e(\bar{z})=\xi_1^e(z)$, we conclude $ \xi_1^e=  constant $. Moreover, the condition $\xi_2^e(\bar{z})=\xi_2^e(z)$ implies $\xi_2^e= constant $ and the condition $\xi_1^o(\bar{z})=\xi_2^o(z)$ implies $\xi_1^o(z)=a_1z$ and $\xi_2^o(z)=b_{-1}\bar{z}$ with $a_1=b_{-1}$, which immediately yield $\xi_1=a_0+a_1z$ and $\xi_2 =a_1\bar{z}+b_0$ for some $a_0,a_1,b_0 \in \C $. It is important to note that, we obtain the conditions $|\xi_1|^2+|\xi_2|^2=2$ and $|\xi_2^e+\xi_1^o|^2+| \xi_2^o+\xi_1^e|^2=2$ by reducing the conditions listed in \eqref{Im-cond2} in Remark \ref{rem-1} 
	and hence we get the following relations between $a_0$, $a_1$ and $b_0$
	$$ |a_0|^2+|a_1|^2=1,|b_0|^2+|a_1|^2=1,a_0\bar{a_1}+a_1\bar{b}_0=0. $$
	This completes the proof.
\end{proof}

\subsection{Model Space Preserving Conjugations}
In this subsection, we characterize $M_{z^2}$-commuting and $M_{z}$, $M_{\bar{z}}$-intertwining conjugations preserving model spaces in terms of $\mathscr{C}_\theta$. The following two propositions will be useful to prove our main results.

\begin{ppsn}\label{pp1}
	If $C$ is a conjugation on $L^2$ such that $CM_{z^2}=M_{\bar{z}^2}C$, then $M_{z}CM_{\bar{z}}$ is also a conjugation on $L^2$ intertwining $M_{z^2}, M_{\bar{z}^2}$, and $M_{z}CM_{\bar{z}}=M_{[z^2\psi_2 , z^2\psi_1 ]}J$ for some $\psi_1,\psi_2 \in L^\infty$ satisfying $\psi_1^o= \psi_2^o$ and $|\psi_1|^2+|\psi_2|^2=2$. 
\end{ppsn}
\begin{proof}
	Suppose $C$ is a conjugation on $L^2$ such that $CM_{z^2}=M_{\bar{z}^2}C$, then  by Corollary \ref{cr4.1},  there exist  $\psi_1,\psi_2 \in L^\infty$ such that $C=M_{[\psi_1,\psi_2]}J$, where $\psi_1^o= \psi_2^o$ and $|\psi_1|^2+|\psi_2|^2=2$. It is easy to see that $M_{z}CM_{\bar{z}}$	is a conjugation on $L^2$ and $M_{z}CM_{\bar{z}}M_{z}=M_{\bar{z}}M_{z}CM_{\bar{z}}$. Therefore,  by applying Corollary \ref{cr4.1} again, we have $M_{z}CM_{\bar{z}}=M_{[\vartheta_1,\vartheta_2]}J$ for some $\vartheta_1,\vartheta_2\in L^\infty $. On the other hand, the following computation for any $f(=f^e\oplus f^o)\in L^2$ yields
	\begin{align*}
		M_{z}CM_{\bar{z}}(f)  =  M_{z}M_{[\psi_1,\psi_2]}JM_{\bar{z}}(f^e\oplus f^o)
		=z^2\psi_2 \bar{f}^e + z^2\psi_1 \bar{f}^o
		= M_{[z^2\psi_2 , z^2\psi_1 ]}J(f),
	\end{align*}
	and hence $ M_{z}CM_{\bar{z}} = M_{[z^2\psi_2 , z^2\psi_1 ]}J$. This completes the proof.
\end{proof}
\begin{ppsn}\label{ppsn-2}
	Let  $C$ be a conjugation on $L^2$ intertwining $M_{z^2}$ and $M_{\bar{z}^2}$.  Then
	$CJ=M_{[\psi_1,\psi_2]} $ for some $\psi_1,\psi_2 \in L^\infty$ satisfying
	$\psi_1^o= \psi_2^o$ and $|\psi_1|^2+|\psi_2|^2=2$. Furthermore, we also have $M_{z}CJM_{\bar{z}}=M_{[\psi_2,\psi_1]}$.
\end{ppsn}
\begin{proof}
	Since $CM_{z^2}=M_{\bar{z}^2}C$, then by Corollary \ref{cr4.1},  $C=M_{[\psi_1,\psi_2]}J $ for some $\psi_1,\psi_2 \in L^\infty$ with
	$\psi_1^o= \psi_2^o$ and $|\psi_1|^2+|\psi_2|^2=2$. Note that $J$ is also a conjugation on $L^2$ and hence $CJ=M_{[\psi_1,\psi_2]}$. For the second part, we proceed similarly as  Proposition \ref{pp1} and obtain the following
	\begin{align*}
		M_{z}CJM_{\bar{z}}(f)  = M_{z}M_{[\psi_1,\psi_2]}M_{\bar{z}}(f^e\oplus f^o)
		= \psi_2 f^e + \psi_1 f^o 
		= M_{[\psi_2,\psi_1]}(f).
	\end{align*}
	Consequently, $M_{z}CJM_{\bar{z}}=M_{[\psi_2,\psi_1]} $.  This completes the proof.
\end{proof}
Now, we are in a position to prove one of our main results related to  \cite[Theorem 4.2]{CAMARA2}. 
\begin{thm}\label{th5.5}
	Let $\alpha, \theta $ be two non-constant inner functions and consider a $M_{z^2},M_{\bar{z}^2}$-intertwining conjugation $C$ on $L^2$. If $C$ satisfies the following conditions
	\begin{enumerate}[(i)]
		\item $C(K_\alpha) \subset K_\theta $ and also, $M_zCM_z(K_\alpha) \subset K_\theta $,
		\item $Re[(M_{z}CJM_{\bar{z}})^*(CJ)] =I$,
	\end{enumerate} 
	then there exists an inner function $\gamma$ such that $C+M_zCM_z = \sqrt{2}\mathscr{C}_\gamma$, where $\gamma$ is divisible by $\alpha$ and $\gamma$ divides $\theta$, in other words, $\alpha \leq \gamma \leq \theta$.
\end{thm}
\begin{proof}
	Since $CM_{z^2}=M_{\bar{z}^2}C$, then by Proposition \ref{ppsn-2} ,  $C=M_{[\psi_1,\psi_2]}J$ and $M_{z}CJM_{\bar{z}}=M_{[\psi_2,\psi_1]} $ for some $\psi_1,\psi_2 \in L^\infty$ satisfying $\psi_1^o= \psi_2^o$ and $|\psi_1|^2+|\psi_2|^2=2$. Next, we consider the reproducing kernel function at $0$, that is $k^\alpha_0= 1-\overline{\alpha(0)}\alpha$ and its conjugate $\tilde{k}^\alpha_0 =\mathscr{C}_\alpha k^\alpha_0 =\bar{z}( \alpha -\alpha(0))$ in $K_\alpha$. Therefore by simple calculations, we get
	\begin{align*}
		C(\tilde{k}^\alpha_0) = M_{[\psi_1,\psi_2]}J(\bar{z}( \alpha -\alpha(0)))
		=  M_{[\psi_1,\psi_2]}(zf^e+zf^o)
		=  \psi_1 zf^o + \psi_2 zf^e,
	\end{align*}
	where $f=\overline{(\alpha -\alpha(0))} $. Similarly, we also have
	\begin{align*}
		M_z C M_z(\tilde{k}^\alpha_0)  = M_z M_{[\psi_1,\psi_2]}J( \alpha -\alpha(0))
		=  M_z M_{[\psi_1,\psi_2]}\overline{(\alpha -\alpha(0))}
		=  \psi_1 zf^e + \psi_2 zf^o.
	\end{align*}
	Since $C(\tilde{k}^\alpha_0),M_zCM_z(\tilde{k}^\alpha_0) \in K_\theta $, then $C(\tilde{k}^\alpha_0)+M_zCM_z(\tilde{k}^\alpha_0)\in K_\theta$. In other words,  $(\psi_1 +\psi_2)zf \in K_\theta $ which implies that there exist $g\in K_\theta $ such that $g=(\psi_1 +\psi_2)z\overline{(\alpha -\alpha(0))}=(\psi_1 +\psi_2)z\bar{\alpha} (1-\bar{\alpha}(0)\alpha )$. Note that $(1-\overline{\alpha(0)}\alpha)^{-1}$ is a bounded analytic function and hence
	\begin{equation*}
		z\bar{\alpha}(\psi_1 +\psi_2) = g(1-\overline{\alpha(0)}\alpha)^{-1} \in H^2,
	\end{equation*}
	which yields $z(\psi_1+\psi_2)\in H^\infty$. Our next aim is to show that $\dfrac{z(\psi_1+\psi_2)}{2}=\gamma(say)$ is an inner function.
	Since $CJ=M_{[\psi_1,\psi_2]}$ and $M_{z}CJM_{\bar{z}}=M_{[\psi_2,\psi_1]} $, then it follows that $$(M_{z}CJM_{\bar{z}})^*(CJ) = M_{[\psi_2,\psi_1]}^*M_{[\psi_1,\psi_2]}.$$
	By considering the decomposition $L^2(\T)=\hcl_1(1,z^2) \oplus \hcl_2(z,z^2)$, we have the following block-matrix representation 
	\begin{align*}
		M_{[\psi_2,\psi_1]}^*M_{[\psi_1,\psi_2]} &=
		\begin{bmatrix}
			M_{\bar{\psi_2^e}\psi_1^e+\bar{\psi_2^o}\psi_1^o}&& M_{\bar{\psi_2^e}\psi_2^o+\bar{\psi_2^o}\psi_2^e}\\
			\\ 
			M_{\bar{\psi_1^o}\psi_1^e+\bar{\psi_1^e}\psi_1^o} && M_{\bar{\psi_1^o}\psi_2^o+\bar{\psi_1^e}\psi_2^e}\\
		\end{bmatrix}.
	\end{align*}
	Therefore, the condition $Re[(M_{z}CJM_{\bar{z}})^*(CJ)] =I$ boils down to $ Re\left[M_{[\psi_2,\psi_1]}^*M_{[\psi_1,\psi_2]}\right]=I$ which produces the following two conditions
	\begin{equation}\label{eq5.2}
		\begin{cases}
			& Re[\bar{\psi_2^e}\psi_1^e+\bar{\psi_2^o}\psi_1^o] =1 ,  \\ 
			& \bar{\psi_2^e}\psi_2^o+\bar{\psi_2^o}\psi_2^e + \psi_1^o\bar{\psi_1^e}+\psi_1^e\bar{\psi_1^o}=0. 
		\end{cases}
	\end{equation}
	Recall that, to characterize the class of conjugations $C$ intertwining $M_{z^2}$ and $M_{\bar{z}^2}$, we initially obtain the conditions \eqref{im-co}  in Remark \ref{rem-4.5}.
	Therefore, combining the conditions \eqref{eq5.2} and $\bar{\psi_2^o}\psi_1^e+\bar{\psi_2^e}\psi_1^o=0$ in \eqref{im-co}, we get
	\begin{align*}
		2Re[\psi_2\bar{\psi_1]} =\psi_2\bar{\psi_1}+\psi_1\bar{\psi_2}
		= 2Re[\bar{\psi_2^e}\psi_1^e+\bar{\psi_2^o}\psi_1^o] + 2Re[\bar{\psi_2^o}\psi_1^e+\bar{\psi_2^e}\psi_1^o]=2.
	\end{align*}
	Thus, $|\psi_1+\psi_2|^2=|\psi_1|^2+|\psi_2|^2+2Re[\psi_2\bar{\psi_1]}=4$, and hence $|\gamma|=|\dfrac{z(\psi_1+\psi_2)}{2}|=1$ a.e. on $\T$. 
	Therefore, $\bar{\alpha}\gamma\in H^2$ shows that $\gamma$ is divisible by $\alpha$.
	Moreover, we also obtain
	\begin{align*}
		\mathscr{C}_\theta (C+M_zCM_z)(k^\alpha_0) & = \mathscr{C}_\theta (M_{[\psi_1,\psi_2]}J+M_zM_{[\psi_1,\psi_2]}JM_z)(1-\overline{\alpha(0)}\alpha)\\
		& = \theta \overline{z(\psi_1+\psi_2)}(1-\overline{\alpha(0)}\alpha)\in K_\theta. 
	\end{align*}
	Therefore, $\theta \bar{\gamma} \in H^2$ implies that $\theta$ is divisible by $\gamma$. Moreover, one can easily conclude that $C+M_zCM_z=2\mathscr{C}_\gamma$. 
\end{proof}
Next, we discuss the relationship between $only~M_{z^2}-commuting$ conjugations and model spaces. Before we start,  recall some useful facts obtained in \cite{Cima} as follows
\begin{equation}
	J^\star(K_\alpha)=K_{\alpha^\#},~
	J^\star \mathscr{C}_\alpha = \mathscr{C}_{\alpha^\#}J^\star,~\quad \text{and} \quad 
	J^\star(\alpha H^2) = \alpha^\# H^2.
\end{equation}
Note that, for an antilinear operator $T$, $T^\sharp$  denotes the antilinear adjoint of $T$.	
\begin{thm}\label{th5.6}
	Let $\alpha, \theta $ be two non-constant inner functions and let $C$ be a $only~M_{z^2}-commuting$ conjugation $C$ on $L^2$. If  $C$ satisfies the following conditions
	\begin{enumerate}[(i)]
		\item $J^\star C$ is $\mathscr{C}_\alpha $- symmetric,
		\item $C(K_\alpha) \subset K_\theta $ and $M_zCM_{\bar{z}}(K_\alpha) \subset K_\theta $,
		\item $Re\left[(M_{z}C\mathscr{C}_\alpha JM_{\bar{z}})^\sharp(C\mathscr{C}_\alpha J)\right] =I$,
	\end{enumerate} 
	then there exists an inner function $\psi$ satisfying $\alpha \leq \psi \leq \theta^\#$
	such that $C+M_zCM_{\bar{z}} = J^\star M_{\frac{2\psi}{\alpha}}$.
\end{thm}

\begin{proof}
	Since $C$ is a $only~M_{z^2}-commuting$ conjugation, then by Corollary \ref{cr4.2}, there exist  $\xi_1,\xi_2\in L^\infty$ such that $C=M_{[\xi_1,\xi_2]}J^\star$ and  $\xi_1,\xi_2$ satisfy $|\xi_1|^2+|\xi_2|^2=2,|\xi_2^e+\xi_1^o|^2+| \xi_2^o+\xi_1^e|^2=2,\xi_1^o(\bar{z})=\xi_2^o(z), \xi_1^e(\bar{z})=\xi_1^e(z), \xi_2^e(\bar{z})=\xi_2^e(z)$. Note that $J^\star C=M_{[\xi_1,\xi_2]}^*$ and by hypothesis, it is  $\mathscr{C}_\alpha$- symmetric. Next, we consider the antilinear operator $J^\star C \mathscr{C}_\alpha $ on $L^2$, and we have the following observations
	\begin{enumerate}[(i)]
		\item $(J^\star C \mathscr{C}_\alpha)^2 =J^\star C \mathscr{C}_\alpha J^\star C \mathscr{C}_\alpha =M_{[\xi_1,\xi_2]}^*\mathscr{C}_\alpha M_{[\xi_1,\xi_2]}^*\mathscr{C}_\alpha=M_{[\xi_1,\xi_2]}^*M_{[\xi_1,\xi_2]}=I$,
		\item $\langle J^\star C \mathscr{C}_\alpha(f),J^\star C \mathscr{C}_\alpha(g)\rangle =\langle g,f\rangle $ for any $f,g \in L^2$.
	\end{enumerate}
	Therefore $J^\star C \mathscr{C}_\alpha $ is a conjugation satisfying $J^\star C \mathscr{C}_\alpha M_{z^2} = M_{\bar{z}^2} J^\star C \mathscr{C}_\alpha $, but observe that $J^\star C \mathscr{C}_\alpha M_z \neq M_{\bar{z}}J^\star C \mathscr{C}_\alpha$.  Moreover, using the hypothesis, we get $$ J^\star C \mathscr{C}_\alpha (K_\alpha)=J^\star C (K_\alpha) \subset J^\star(K_\theta)=K_{\theta^{\#}},$$
	and $M_zJ^\star C \mathscr{C}_\alpha M_z(K_\alpha)=J^\star M_zC M_{\bar{z}}\mathscr{C}_\alpha (K_\alpha)=J^\star M_zC M_{\bar{z}}(K_\alpha)\subset J^\star(K_\theta)=K_{\theta^{\#}} $.
	Furthermore, 
	\begin{align*}
		Re[(M_{z}J^\star C \mathscr{C}_\alpha JM_{\bar{z}})^*(J^\star C \mathscr{C}_\alpha J)]  
		= Re[(M_{z}C\mathscr{C}_\alpha JM_{\bar{z}})^\sharp(C\mathscr{C}_\alpha J)] =I.
	\end{align*}
	Finally, by applying Theorem \ref{th5.5} and using the fact $\mathscr{C}_\alpha M_z=M_{\bar{z}}\mathscr{C}_\alpha$ , there exists an inner function $\psi$ satisfying $\alpha \leq \psi \leq \theta ^{\#}$ such that 
	\begin{align*}
		& J^\star C \mathscr{C}_\alpha +M_z J^\star C \mathscr{C}_\alpha M_z =2\mathscr{C}_\psi 
		\implies  J^\star C \mathscr{C}_\alpha + J^\star M_z C  M_{\bar{z}}\mathscr{C}_\alpha = 2\mathscr{C}_\psi \\
		& \implies  C +  M_z C  M_{\bar{z}} =J^\star 2\mathscr{C}_\psi \mathscr{C}_\alpha 
		\implies   C +  M_z C  M_{\bar{z}} = J^\star M_{\frac{2\psi}{\alpha}}.
	\end{align*}
	This completes the proof. 
\end{proof}
It is important to note that assumptions made in Theorem~\ref{th5.5} and Theorem~\ref{th5.6} above are fully consistent with the assumptions made in \cite[Theorem~4.2]{CAMARA2} and  \cite[Theorem~4.6]{CAMARA2}, respectively.

\subsection{$S$-Invariant Subspaces Preserving Conjugations}
$~$
In this subsection, we characterize all conjugations $C$ preserving shift invariant subspaces and intertwining $M_{z^2}$, $M_{\bar{z}^2}$. Moreover, we also characterize all  $M_{z^2}$-commuting conjugations $C$ that preserve shift-invariant subspaces.

\begin{thm}
	There exist no $M_{z^2}$ and $M_{\bar{z}^2}$ intertwining conjugations in $L^2$ that map any Beurling type subspace $\alpha H^2$ into another $\theta H^2$ for any two inner functions $\alpha$ and $\theta $.
\end{thm}
\begin{proof}
	Let $C$ be a $M_{z^2},M_{\bar{z}^2}-intertwining$ conjugation  on $L^2$ such that $C(\alpha H^2)\subset \theta H^2$. Since $C$ satisfies $CM_{z^2}=M_{\bar{z}^2}C$, then by Corollary\eqref{cr4.1}, there exist $\psi_1,\psi_2 \in L^\infty$ such that $C=M_{[\psi_1,\psi_2]}J$ and $\psi_1^o= \psi_2^o$, $|\psi_1|^2+|\psi_2|^2=2$. Since $C(\alpha H^2)\subset \theta H^2$, then there exists $f\in H^2$ such that
	\begin{align*}
		C(\alpha)  = M_{[\psi_1,\psi_2]}J(\alpha)
		= M_{[\psi_1,\psi_2]}J(\alpha^e \oplus \alpha^o )
		= \psi_1\bar{\alpha}^e + \psi_2 \bar{\alpha}^o = \theta f.
	\end{align*}  
	Similarly, for $n=2k,k\geq 0$, we get
	\begin{align*}
		C(\alpha z^n)  = M_{[\psi_1,\psi_2]}J(\alpha z^n)
		= M_{[\psi_1,\psi_2]}J(\alpha^e z^n \oplus \alpha^o z^n)
		= \bar{z}^n(\psi_1\bar{\alpha}^e + \psi_2 \bar{\alpha}^o)=\theta h_n,
	\end{align*}
	for some $h_n\in H^2$  and hence $\bar{z}^n f = h _n
	\in H^2$. Thus for any non-negative even integer $n$, 
	\begin{align*}
		\langle \bar{z}^nf , \bar{z} \rangle =0
		\implies \langle f, z^{n-1} \rangle  =0 
		\implies  f^o =0.
	\end{align*}
	Similarly, $\langle \bar{z}^nf , \bar{z}^2 \rangle =0$ implies $f^e =0$ and hence we conclude that $f=0$. Thus  $C(\alpha)=0$, a contradiction. 
\end{proof}
Now we move to the characterization of $M_{z^2}-commuting$ conjugations that map a shift-invariant subspace $\alpha H^2$ into another one, say $\theta H^2$.
\begin{thm}\label{thm5.8n}
	Let $\alpha$ and $\beta$ be two inner functions,  and let $C$ be any $M_{z^2}$-commuting conjugation on $L^2$ satisfying $Re[(M_{\bar
		z}CJ^\star M_z)^*(CJ^\star)]=I$. If $C(\alpha H^2)\subset \beta H^2$, then $\beta \beta^\# \leq \alpha \alpha^\#$ and $C+M_zCM_z =2\mathscr{C}_{\theta}J^*\mathscr{C}_{\alpha}$, where $\theta $ is an inner function such that  $\beta \leq \theta$, $\theta \theta ^\# =\alpha \alpha^\#$, and $\theta\overline{\alpha ^\#}$ is symmetric. 
\end{thm}	
\begin{proof}
	Since $CM_{z^2}=M_{z^2}C$, then by Corollary\eqref{cr4.1}, there exist $\xi_1,\xi_2\in L^\infty$ such that $C=M_{[\xi_1,\xi_2]}J^\star$ and $|\xi_1|^2+|\xi_2|^2=2,|\xi_2^e+\xi_1^o|^2+| \xi_2^o+\xi_1^e|^2=2,\xi_1^o(\bar{z})=\xi_2^o(z), \xi_1^e(\bar{z})=\xi_1^e(z), \xi_2^e(\bar{z})=\xi_2^e(z)$.
	Now, by using the condition $C(\alpha H^2)\subset \beta H^2$, we conclude that
	$$C(\alpha)=M_{[\xi_1,\xi_2]}J^\star (\alpha) = \xi_1 {\alpha^e}^{\#} +\xi_2 {\alpha^o}^{\#} \in \beta H^2.$$
	Moreover, we also have  $M_z C M_z (\alpha H^2)\subset M_zC(\alpha)\subset M_z(\beta H^2)\subset \beta H^2$. Similarly, 
	$$M_zCM_z(\alpha )= \xi_1 {\alpha^o}^{\#} +\xi_2 {\alpha^e}^{\#} \in \beta H^2.$$
	Therefore $\xi_1 {\alpha^e}^{\#} +\xi_2 {\alpha^o}^{\#}  + \xi_1 {\alpha^o}^{\#} +\xi_2 {\alpha^e}^{\#} \in \beta H^2$ implies that there exists $h\in H^2$ such that 
	\begin{equation}\label{eq-5.4}
		(\xi_1+\xi_2)\alpha^{\#} =\beta h.
	\end{equation}
	On the other hand $CJ^\star = M_{[\xi_1,\xi_2]}$ and $M_{\bar
		z} CJ^\star M_z = M_{[\xi_2,\xi_1]}$, and hence $Re\left[(M_{\bar
		z}CJ^\star M_z)^*(CJ^\star)\right]=I$, which yields the following conditions
	\begin{equation}\label{eq-5.5}
		\begin{cases}
			& Re\left[\bar{\xi_2^e}\xi_1^e+\bar{\xi_2^o}\xi_1^o\right] =1 ,  \\ 
			& \bar{\xi_2^e}\xi_2^o+\bar{\xi_2^o}\xi_2^e + \xi_1^o\bar{\xi_1^e}+\xi_1^e\bar{\xi_1^o}=0. 
		\end{cases}
	\end{equation}
	Next, by combining conditions obtained in  \eqref{eq-5.5} and  \eqref{Im-cond2}, we conclude
	\begin{align*}
		2Re[\xi_2\bar{\xi_1}]
		= 2Re[\bar{\xi_2^e}\xi_1^e+\bar{\xi_2^o}\xi_1^o] + 2Re[\bar{\xi_2^o}\xi_1^e+\bar{\xi_2^e}\xi_1^o] = 2,
	\end{align*}
	and hence $|\xi_1+\xi_2|^2=|\xi_1|^2 +|\xi_2|^2+ 2Re[\xi_2\bar{\xi_1}]=4$. Suppose $\psi = \dfrac{\xi_1 + \xi_2}{2}$, then $|\psi|=1$. Therefore, from \eqref{eq-5.4} we get $\psi\alpha^{\#} = \beta f$ for some $f\in H^2$. On the other hand, $\psi$ is a symmetric function since $\psi(\bar{z})=\frac{1}{2}[\xi_1(\bar{z})+\xi_2(\bar{z})]= \frac{1}{2}[(\xi_1^e(\bar{z})+\xi_1^o(\bar{z})+(\xi_2^e(\bar{z})+\xi_2^o(\bar{z})]=\frac{1}{2}[(\xi_1^e(z)+\xi_2^o(z)) + (\xi_2^e(z)+\xi_1^o(z))]= \frac{1}{2}[(\xi_1+\xi_2)(z)]=\psi(z)$. Let $\theta = \beta f$, then $\theta \overline{\alpha^{\#}}=\psi$ is symmetric. Moreover, $\theta$ is an inner function such that $\beta \leq \theta$ and $\theta \theta^{\#}=\alpha \alpha^{\#}$ which immediately implies $\beta \beta ^{\#} \leq \alpha \alpha^\#$. Furthermore, by simple calculations, one can also obtain the following identity $$C+M_zCM_z =2\mathscr{C}_\theta J^* \mathscr{C}_\alpha .$$ This completes the proof.
\end{proof}

\begin{rmrk}
	Note that the converse of Theorem~\ref{thm5.8n} may not be true in general. In other words, for the converse, if we assume $\beta \beta ^{\#} \leq \alpha \alpha^\#$, then by mimicking the proof of Theorem 5.2 in \cite{CAMARA2}, one can show that $C+M_zCM_z =2\mathscr{C}_\theta J^* \mathscr{C}_\alpha$, which maps $\alpha H^2$ onto $\theta H^2 \subset \beta H^2$ but not necessarily  $C(\alpha H^2) \subset \beta H^2$.
\end{rmrk}

\section{Conjugations related to $M_B$}
In this section, we want to investigate some of the earlier characterization results in the context of $M_B$, where $B$ is a finite Blaschke product.  
Let $B$ be a finite Blaschke product of degree (or order, say) $n$, then the dimension of the model space $K_B:=H^2\ominus BH^2$ is $n$. Let $\{e_1,e_2,\ldots e_n\}$ be an orthonormal basis of $K_B$.  Recall that (see Section 3 \eqref{l2}),
$$\L^2(\T) = \hcl_1 (e_1,B)\oplus \hcl_2 (e_2,B)\oplus \cdots \oplus \hcl_n (e_n,B).$$
For example, if  $B(z)=z^2$, then for any $f\in L^2(\T)$, we have  $f=f^e\oplus f^o$, where $f^e \in \bigvee\limits_{k=-\infty}^{\infty}\{z^{2k}\} =\mathcal{H}_1(1,z^2)$ and $f^o \in \bigvee\limits_{k=-\infty}^{\infty}\{z^{2k+1}\}=\mathcal{H}_1(z,z^2)$, and we are already familiar with such factorization in earlier sections. Now if $B(z)=z^n$,  then for any  $f(z)=\sum\limits_{k=-\infty}^{\infty}a_kz^k \in L^2(\T)$ has the following decomposition 
\begin{equation}\label{eq6.1}
	f(z)  = \sum_{k=-\infty}^{\infty}a_{nk}z^{nk} +\sum_{k=-\infty}^{\infty}a_{nk+1}z^{nk+1}+\cdots +\sum_{k=-\infty}^{\infty}a_{nk+n-1}z^{nk+n-1} ,
\end{equation}
where  $f_i(z)=\sum\limits_{k=-\infty}^{\infty}a_{nk+i-1}z^{nk+i-1}\in \mathcal{H}_i(z^i,z^n)$ for $i\in\{0,1,\ldots ,n-1\}$.
Recall the following two well-known facts:
\begin{equation}\label{last1}
	\mathscr{C}_BM_B = M_{\bar{B}}\mathscr{C}_B \quad \text{and}\quad 
	\mathscr{C}_B^\star M_B = M_B\mathscr{C}_B^\star .
\end{equation}
The following theorem is one of the main results in this section.
\begin{thm}\label{th6.1}
	Let $C:L^2(\T) \to L^2(\T)$ be a conjugation.
	\begin{enumerate}[(i)]
		\item If $CM_B = M_{\bar{B}}C$, then there exist $\phi_1,\phi_2,\ldots ,\phi_n \in L^{\infty}$ such that $C=M_{[\phi_1,\phi_2,\ldots ,\phi_n]} \mathscr{C}_B $, where $M_{[\phi_1,\phi_2,\ldots ,\phi_n]}$ is unitary and $\mathscr{C}_B$-  symmetric.
		\item If $CM_{B} =M_BC$, then there exist $\psi_1,\psi_2,\ldots ,\psi_n \in L^{\infty}$ such that $C=M_{[\psi_1,\psi_2,\ldots ,\psi_n]} \mathscr{C}_B^\star $, where $M_{[\psi_1,\psi_2,\cdots ,\psi_n]}$ is unitary and $\mathscr{C}_B^\star$- symmetric.
	\end{enumerate} 
\end{thm}
\begin{proof}
	\begin{enumerate}[(i)]
		\item Let $C:L^2(\T) \to L^2(\T)$ be a conjugation such that $CM_B = M_{\bar{B}}C$, then  using  \eqref{last1}, we conclude that $C\circ \mathscr{C}_B$ is a unitary operator on $L^2$ satisfying $C\circ \mathscr{C}_B M_B =M_B C\circ \mathscr{C}_B$. Therefore, from Theorem \ref{comm}, we get $\phi_1, \phi_2, \ldots , \phi_n\in L^\infty$ such that $C\circ \mathscr{C}_B = M_{[\phi_1,\phi_2,\cdots ,\phi_n]}$, and hence $C=M_{[\phi_1,\phi_2,\ldots ,\phi_n]}\mathscr{C}_B$.  It is easy to observe that $M_{[\phi_1,\phi_2,\ldots ,\phi_n]}$ is a unitary operator  and
		\begin{equation*}
			M_{[\phi_1,\phi_2,\ldots ,\phi_n]}\mathscr{C}_B M_{[\phi_1,\phi_2,\ldots ,\phi_n]}\mathscr{C}_B =I,
		\end{equation*}
		and hence
		\begin{equation*}
			\mathscr{C}_B M_{[\phi_1,\phi_2,\ldots ,\phi_n]}\mathscr{C}_B =	M^*_{[\phi_1,\phi_2,\ldots ,\phi_n]}
		\end{equation*}	
		implying that $M_{[\phi_1,\phi_2,\ldots ,\phi_n]}$ is $\mathscr{C}_B$-symmetric.
		\item Let $C$ be a conjugation on $L^2$ such that  $CM_B=M_BC$. Then again, using  \eqref{last1}, we conclude that the unitary operator $C\circ \mathscr{C}_B^\star$ also commutes with $M_B$, and hence by Theorem \ref{comm}, there exist $\psi_1,\psi_2, \ldots , \psi_n \in L^\infty$ such that $$C \circ \mathscr{C}^\star_B = M_{[\psi_1,\psi_2,\cdots ,\psi_n]}$$ 
		which implies $C = M_{[\psi_1,\psi_2,\ldots ,\psi_n]}\mathscr{C}^\star_B $. Similarly, as in (i), $M_{[\psi_1,\psi_2,\cdots ,\psi_n]}$ is unitary and $\mathscr{C}_B^\star$- symmetric.
	\end{enumerate}
\end{proof}
Now we turn our discussion to find the relations between $M_{B}$-commuting conjugations and model spaces for $B(z)=z^n$.
Observe that the actions of $\mathscr{C}_{B}$ and $\mathscr{C}_B^\star$ are same on $K_B$. The action of $\mathscr{C}_{z^n}$ on the set $\{1,z,\ldots,z^{n-1}\}$ is given by $\mathscr{C}_{z^n}(z^k)= z^{(n-1)-k}$ for $0\leq k \leq n-1$. Depending on $n$, we have the following two cases
\begin{enumerate}
	\item If $n$ is odd, then the action of $\mathscr{C}_{z^n}$ is  the following
	\begin{align}
		\begin{tabular}{ |c| c| c| c| c| c| c| c| c|c|c|} 
			\hline
			$1$ & $z$ & $z^2$ & $\cdots$ & $z^{\frac{n-3}{2}}$ & $z^{\frac{n-1}{2}}$ & $z^{\frac{n+1}{2}}$ & $\cdots$ & $z^{n-3}$ & $z^{n-2}$ & $z^{n-1}$ \\[2ex]
			\hline
			$z^{n-1}$ & $z^{n-2}$ & $z^{n-3}$  & $\cdots$ & $z^{\frac{n+1}{2}}$ & $z^{\frac{n-1}{2}}$ & $z^{\frac{n-3}{2}}$ & $\cdots$ & $z^2$ & $z$ & $1$ \\[2ex]
			\hline
		\end{tabular},
	\end{align}
	where the first row contains the basis elements and the second row contains their images, respectively. 
	\item If $n$ is even, then as above the action of $\mathscr{C}_{z^n}$ is as follows:
	\begin{align}
		\begin{tabular}{ |c| c| c| c| c| c| c| c|c|c|} 
			\hline
			$1$ & $z$ & $z^2$ & $\cdots$ &  $z^{\frac{n}{2}-1}$ & $z^{\frac{n}{2}}$ & $\cdots$ & $z^{n-3}$ & $z^{n-2}$ & $z^{n-1}$ \\[2ex]
			\hline
			$z^{n-1}$ & $z^{n-2}$ & $z^{n-3}$  & $\cdots$ & $z^{\frac{n}{2}}$ & $z^{\frac{n}{2}-1}$ & $\cdots$ & $z^2$ & $z$ & $1$ \\[2ex]
			\hline
		\end{tabular}.
	\end{align}
\end{enumerate}
Also,  note that $\mathscr{C}_{z^n}=M_{z^{n-1}}J$, and 
\begin{equation}\label{eq6.4}
	\mathscr{C}^\star_{z^n}=
	\begin{cases*}
		M_{[z^{n-1},z^{n-3},\ldots ,z^2,1,\bar{z}^2, \ldots ,\bar{z}^{n-3},\bar{z}^{n-1}]} J^\star , ~\text{if~n~is ~odd,}\\
		M_{[z^{n-1},z^{n-3},\ldots ,z,\bar{z}, \ldots ,\bar{z}^{n-3},\bar{z}^{n-1}]} J^\star , ~\text{if ~n ~is ~even}.
	\end{cases*}
\end{equation}
The above two relations follow from the decomposition of  $f\in L^2$ mentioned in \eqref{eq6.1} along with the action of $\mathscr{C}_{z^n}$ and $\mathscr{C}^\star_{z^n}$, respectively. Now we are in a position to state and prove our second main results in this section.
\begin{thm}
	Let $C$ be a conjugation on $L^2$  such that $CM_{z^n} =M_{z^n} C$ and let $C$ preserve the Hardy space $H^2$, then $C = M_{[\xi_1,\xi_2, \ldots ,\xi_n]}J^\star$, where $\xi_1 ,z\xi_2,z^2\xi_3, \ldots ,z^{n-1}\xi_n \in H^\infty $ and $M_{[\xi_1,\xi_2, \ldots ,\xi_n]}$ is a $J^\star$- symmetric unitary operator.
\end{thm}

\begin{proof}
	By using Theorem\ref{th6.1}, there exist $\psi_1, \psi_2 , \ldots , \psi_n \in 
	L^\infty$ such that $C=M_{[\psi_1,\psi_2, \ldots ,\psi_n]}C^\star_{z^n}$. Now from \eqref{eq6.4}, we get
	\begin{equation*}
		C=
		\begin{cases*}
			M_{[\psi_1,\psi_2, \ldots ,\psi_n]}M_{[z^{n-1},z^{n-3},\ldots ,z^2,1,\bar{z}^2, \ldots ,\bar{z}^{n-3},\bar{z}^{n-1}]} J^\star ,~\text{if~n~is ~odd},\\
			M_{[\psi_1,\psi_2, \ldots ,\psi_n]}M_{[z^{n-1},z^{n-3},\ldots ,z,\bar{z},\bar{z}^3, \ldots ,\bar{z}^{n-3},\bar{z}^{n-1}]} J^\star , ~\text{if ~n ~is ~even}. 
		\end{cases*}
	\end{equation*}
	Next, we divide our analysis into the following two cases. 
	\begin{enumerate}[(i)]
		\item \textbf{Let $n$ be an odd natural number.}
		
		Now we rewrite $C$ as follows
		\begin{align*}
			C & = M_{[\xi_1,\xi_2, \ldots ,\xi_n]}J^\star,
		\end{align*}
		where $\xi_1 = z^{n-1}\psi_n, \xi_2 = z^{n-3}\psi_{n-1} ,\ldots ,\xi_{\frac{n-1}{2}}= z^2 \psi_{\frac{n+3}{2}} ,\xi_{\frac{n+1}{2}} = \psi_{\frac{n+1}{2}}  ,\xi_{\frac{n+3}{2}} = \bar{z}^2\psi_{\frac{n-1}{2}}, \ldots$, $ \xi_n = \bar{z}^{n-1}\psi_1 $.
		Therefore by proceeding similarly as in Theorem\ref{th6.1}, we conclude that $M_{[\xi_1,\xi_2, \ldots ,\xi_n]}$ is a unitary operator and $J^\star$- symmetric. Since $C(H^2)\subset H^2$, then we conclude
		$\xi_1 ,z\xi_2,$ $z^2\xi_3, \ldots ,z^{n-1}\xi_n \in H^\infty$. 
		
		\item \textbf{Let $n$  be an even natural number.}
		
		Like case (i), we again rewrite $C$ as follows
		\begin{align*}
			C & = M_{[\xi_1,\xi_2, \ldots ,\xi_n]}J^\star,
		\end{align*}
		where $\xi_1 = z^{n-1}\psi_n, \xi_2 = z^{n-3}\psi_{n-1} ,\ldots ,\xi_{\frac{n}{2}}= z \psi_{\frac{n}{2}+1} ,\xi_{\frac{n}{2}+1} = \bar{z}\psi_{\frac{n}{2}}  ,\ldots$, $ \xi_n = \bar{z}^{n-1}\psi_1 $. Similarly, using the  invariance property $C(H^2)\subset H^2$, we conclude $\xi_1 ,z\xi_2,z^2\xi_3, \ldots ,z^{n-1}\xi_n \in H^\infty $.
	\end{enumerate}
	\vspace{0.1in}
	
	\noindent Therefore in both  cases $C  = M_{[\xi_1,\xi_2, \ldots ,\xi_n]}J^\star$, where $\xi_1 ,z\xi_2,z^2\xi_3, \ldots ,z^{n-1}\xi_n \in H^\infty $. Moreover, it is easy to check that $M_{[\xi_1,\xi_2, \ldots ,\xi_n]}$ is a unitary operator and  $J^\star$- symmetric.
\end{proof}
\begin{thm}
	There is no conjugation $C$ on $L^2$ for which $M_{z^n}$ is $C-$symmetric and $C(H^2)\subset H^2$.
\end{thm}
\begin{proof}
	Let  $C$ be a conjugation on $L^2$ such that $CM_{z^n}C =M_{\bar{z}^n} $ and $C(H^2)\subset H^2$.
	Therefore by applying Theorem~\ref{th6.1} we get
	\begin{align}\label{finsec1}
		C  = M_{[\phi_1,\phi_2,\ldots ,\phi_n]} \mathscr{C}_{z^n}
		= M_{[\phi_1,\phi_2,\ldots ,\phi_n]}M_{z^{n-1}}J
		= M_{[\zeta_1,\zeta_2,\ldots ,\zeta_n]}J,
	\end{align} 
	where $\zeta_1=z^{n-1} \phi_n ,\zeta_2 = z^{n-1} \phi_1 ,\zeta_3 = z^{n-1} \phi_2, \ldots ,\zeta_n =z^{n-1}\phi_{n-1} .$
	Moreover, the fact $C(H^2)\subset H^2$ implies  $\zeta_1,\zeta_2, \ldots , \zeta_n \in H^\infty$.
	Now by using \eqref{finsec1} we get for any $k\geq 0$
	\begin{align*}
		\langle C(z^{nk}),\bar{z}^j \rangle  =0\implies \langle \zeta_1, z^{nk-j} \rangle  =0\quad \text{for}\quad 1\leq j\leq n,
	\end{align*}
	and hence $\zeta_1 =0$. By mimicking a similar kind of argument, we conclude  $\zeta_{i} = 0$ for $2\leq i \leq n$. 
	On the other hand, $ M_{[\zeta_1,\zeta_2,\ldots ,\zeta_n]}=CJ$ is an unitary operator, a contradiction.  This completes the proof.
\end{proof}
We end the article with the following interesting question:
\begin{prob}
	What is the complete characterization for the class of conjugations that either commute with $M_B$ or intertwine the operators $M_B$ and $M_{\overline{B}}$ together with the invariant property of various subspaces of $L^2$ corresponding to any arbitrary finite Blaschke product $B$?
\end{prob}



\section*{Acknowledgments}
\textit{ The authors would like to
	thank Chandan Pradhan for the helpful discussions. We would also like to thank Dr. Srijan Sarkar for introducing this area to us. The research of the first-named author is supported by the Mathematical Research Impact Centric Support (MATRICS) grant, File No: MTR/2019/000640, by the Science and Engineering Research Board (SERB), Department of Science $\&$ Technology (DST), Government of India. The second named author gratefully acknowledges the support provided by IIT Guwahati, Government of India.}

\end{document}